\documentclass[11pt]{article}
\usepackage{amssymb}
\usepackage{amsmath}
\usepackage{amsthm}
\usepackage{latexsym}
\usepackage[english]{babel}
\usepackage{hyperref}
\usepackage{verbatim}

\usepackage[usenames, dvipsnames]{color} 

\allowdisplaybreaks
\newtheorem{theorem}{Theorem}[section]
\newtheorem{definition}{Definition}[section]
\newtheorem{lemma}{Lemma}[section]

\newtheorem{remark}{Remark}[section]
\newtheorem{example}{Example}[section]
\newtheorem{corollary}{Corollary}[section]

\let\oldbibliography\thebibliography
\renewcommand{\thebibliography}[1]{%
 \footnotesize
  \oldbibliography{#1}%
  \setlength{\itemsep}{0pt}%
}

\def \de {\partial}

\def \N {\mathbb{N}}

\def \la {\langle}
\def \ra {\rangle}

\def \R {\mathbb{R}}

\def \H {\mathbb{H}}

\def \na {\nabla}

\newcommand{\Hn}{{\mathbb{H}^n}}

\begin{document}

\title{A characterization of gauge balls in $\Hn$\\ by horizontal curvature}

\author{Chiara Guidi$^{(1)}$ \& Vittorio Martino$^{(2)}$ \& Giulio Tralli$^{(3)}$}
\addtocounter{footnote}{1}
\footnotetext{Dipartimento di Matematica, Universit\`a di Bologna, piazza di Porta S.Donato 5, 40126 Bologna, Italy. E-mail address:
{\tt{chiara.guidi12@unibo.it}}}
\addtocounter{footnote}{1}
\footnotetext{Dipartimento di Matematica, Universit\`a di Bologna, piazza di Porta S.Donato 5, 40126 Bologna, Italy. E-mail address:
{\tt{vittorio.martino3@unibo.it}}}
\addtocounter{footnote}{1}
\footnotetext{Dipartimento d'Ingegneria Civile e Ambientale (DICEA), Universit\`a di Padova, Via Marzolo 9, 35131 Padova, Italy. E-mail address: {\tt{giulio.tralli@unipd.it}}}

\date{}
\maketitle

\vspace{5mm}

{\noindent\bf Abstract} {\small In this paper we aim at identifying the level sets of the gauge norm in the Heisenberg group $\Hn$ via the prescription of their (non-constant) horizontal mean curvature. We establish a uniqueness result in $\H^1$ under an assumption on the location of the singular set, and in $\H^n$ for $n\geq 2$ in the proper class of horizontally umbilical hypersurfaces.
}

\vspace{5mm}


\vspace{4mm}


\vspace{4mm}

\section{Introduction}

\noindent
If we identify the Heisenberg group $\Hn$ with $\R^{2n+1}=\R^{n}\times\R^n\times \R$ with generic point $\xi=(x,y,t)$ and we choose the group law
\begin{equation}\label{glaw}
\xi\circ \xi'=(x,y,t)\circ(x',y',t')=\left(x+x',y+y',t+t'+2\sum_{k=1}^n (x_k y'_k  - y_k x'_k)\right),
\end{equation}
the so-called homogeneous gauge is the function defined by
\begin{equation*}
\rho(\xi) =\left((|x|^2+|y|^2)^2+t^2\right)^{\frac{1}{4}}.
\end{equation*}
Such a $\rho(\cdot)$ is in fact homogeneous of degree $1$ with respect to the family of dilations
\begin{equation}\label{defdil}
\delta_R(\xi)=(R x, R y, R^2 t), \qquad R>0,
\end{equation}
and it provides the defining function of the following \emph{gauge balls} (sometimes called Kor\'anyi balls)
\begin{equation}\label{defballs}
B_R(\xi_0)=\left\{ \xi \in \Hn \; : \; \rho(\xi_0^{-1}\circ\xi) < R \right\},\qquad\mbox{ for }\xi_0\in\Hn,\, R>0.
\end{equation}
The gauge function appeared in \cite{KV} in the study of singular integrals on homogeneous spaces. It has played over the years a crucial role in the analysis of pdes of sub-elliptic type since the discovery in \cite[Theorem 2]{Fo73} that $\rho^{-2n}(\cdot)$ is, up to a constant, the fundamental solution of the Heisenberg subLaplacian $\Delta_{\Hn}$. It is in fact known since \cite[Th\'eor\`eme 3]{Ga77} (see also the treatment in \cite[Section 5]{BLU}) the validity of an analogue of the classical Gauss-Koebe theorem saying that the pointwise value of every solution $u$ to $\Delta_{\Hn}u=0$ can be represented as a weighted average of the values of $u$ on gauge balls $B_R$. The weight is given by the squared norm of the horizontal gradient of $\rho$ (which is homogeneous of degree $0$ but not constant). Gauge balls are actually characterized by such a weighted mean value property for $\Delta_{\Hn}$-harmonic functions as proved by Lanconelli in \cite{L13}.\\
The metric balls $B_R$ defined in \eqref{defballs} are not the unique choice of ``balls'' adapting to the subRiemannian features of the Heisenberg group. For instance, the Carnot-Carath\'edory balls play somehow the role of the geodesic balls in $\Hn$. Furthermore it is very much related to our purposes the case of the domains bounded by the so-called \emph{Pansu spheres}: they are the cmc-spheres with respect to the relevant notion of horizontal mean curvature (see Definition \ref{defcurv} below) and they are the conjectured unique minimizers for the isoperimetric inequality \cite{PP}. The solution of the isoperimetric problem in the Heisenberg group, also known as Pansu's conjecture, has generated a great amount of attention and several proofs appeared in the literature under extra-assumptions on the class of competitors \cite{LM, Cap, DGN08, Mo, MoRi, Ri}. Concerning the related Alexandrov-type problem, it was shown in \cite{RiRo} that Pansu spheres are the only rotationally invariant hypersurfaces with constant horizontal mean curvature. To the best of our knowledge, a result which is reminiscent of the classical Alexandrov theorem \cite{A} is available only in $\mathbb{H}^1$: as a matter of fact, in \cite[Theorem 6.10]{RiRo1} Ritor\'e and Rosales proved that Pansu spheres in $\mathbb{H}^1$ are the only $C^2$-smooth critical points of the horizontal perimeter under volume constraint. For $n\geq 2$, various characterizations of Pansu spheres among horizontally umbilical hypersurfaces were established in \cite{CCHY18}.

\vskip 0.4cm

In this paper we take a new perspective as we address the question of characterizing the gauge balls by prescribing the horizontal mean curvature. In a similar spirit, in a companion paper \cite{MTsw} two of us have dealt with various characterizations of gauge balls through suitable overdetermined problems. To give a better description of the main results we provide the reader with some initial background on the main notions involved, and we refer to Section \ref{sec2} for the precise definitions. In $\Hn$ the horizontal distribution is spanned at any point $\xi=(x,y,t)$ by the vector fields
\begin{equation*}
X_j=\frac{\partial}{\partial x_j}-2y_j\frac{\partial}{\partial t},\quad  Y_j=\frac{\partial}{\partial y_j}+2x_j\frac{\partial}{\partial t}, \quad j=1,\dots,n,
\end{equation*}
which are left-invariant with respect to the group law \eqref{glaw} and homogeneous of degree one with respect to \eqref{defdil}. In our notations we let
$$\mathcal{H}_{\xi}=span\{X_1, \ldots, X_n, Y_1, \ldots, Y_n\}.$$
We also denote $T=\frac{\partial}{\partial t}$, and we consider in $\Hn$ the Riemannian metric $\la \cdot, \cdot \ra$ which makes the basis $\mathcal{B}=\{X_1, \ldots, X_n, Y_1, \ldots, Y_n, T\}$ orthonormal. If we consider a smooth hypersurface $M\subset \Hn$, a point $\xi\in M$ is said to be characteristic if the tangent space of $M$ at $\xi$ coincides with $\mathcal{H}_{\xi}$. At any point $\xi\in M$ which is not characteristic it is thus well-defined the so-called horizontal normal $\nu^H$ as the normalized $\la \cdot, \cdot \ra$-orthogonal projection on $\mathcal{H}_{\xi}$ of the metric (outer, whenever possible) unit normal $\nu$. The horizontal mean curvature is the divergence of such $\nu^H$ (see Section \ref{sec2} for the precise definitions), which is therefore well-defined at any non-characteristic point. A simple computation shows that the horizontal mean curvature of $\de B_R(0)\subset \Hn$ is proportional to the distance to the $t$-axis, i.e. it is a constant multiple of $\sqrt{|x|^2+|y|^2}$ at any point $(x,y,t)\in \de B_R(0)$ (outside of the two poles sitting on the $t$-axis, which correspond to the only characteristic points for the gauge sphere). In $\H^1$ our main result reads as follows

\begin{theorem}\label{main}
Let $M$ be a smooth surface in $\H^1$ which is connected, orientable, compact, and without boundary. Assume that there are no characteristic points of $M$ outside of the line $\{(0,0,t)\in\mathbb{H}^1\,:\, t\in\R\}$. If at every non-characteristic point $(x,y,t)\in M$ the horizontal mean curvature of $M$ is proportional to $\sqrt{x^2+y^2}$ up to a constant factor $c\neq 0$, then $c>0$ and there exists $t_0\in\R$ such that $M=\de B_R(\xi_0)$ with $R=\sqrt{\frac{3}{c}}$ and $\xi_0=(0,0,t_0)$.
\end{theorem}

The restriction to the ($n=1$)-dimensional case in the previous theorem relies on the fact that the $2$-dimensional surface $M\subset\mathbb{H}^1$ has only one horizontal tangent vector field at every non-characteristic point, and $M$ is then `ruled' by its integral curves (as it is clear from the analysis developed in \cite{CHMY, RiRo1}). In higher dimensions we have the following counterpart, which is a characterization of gauge spheres under the proper prescribed curvature assumption among the class of umbilic hypersurfaces introduced in \cite{CCHY18} (see Definition \ref{def: umbilical} below).

\begin{theorem}\label{mainumb}
Fix $n\geq 2$. Let $M$ be a smooth hypersurface of $\Hn$ which is connected, orientable, compact, and without boundary. Suppose that $M$ is umbilic and that, at every non-characteristic point $(x,y,t)\in M$, the horizontal mean curvature of $M$ is proportional to $\sqrt{|x|^2+|y|^2}$ up to a constant factor $c\neq 0$. Then $c>0$ and there exists $t_0\in\R$ such that $M=\de B_R(\xi_0)$ with $R=\sqrt{\frac{1}{c}\frac{2n+1}{2n-1}}$ and $\xi_0=(0,0,t_0)$.
\end{theorem}

The paper is organized as follows. In Section \ref{sec2} we recall the main definitions involved and we show some basic properties. In Section \ref{sec3} we give the proofs of Theorem \ref{main} and Theorem \ref{mainumb} which follow a similar pattern: the main aim is to infer, under the respective assumptions, that the key functions $\varphi_h, \varphi_v$ introduced below in \eqref{defphis} and \eqref{defphisn} are constant throughout $M$. Finally, we will show in Corollary \ref{corcyl} that Theorems \ref{main} and \ref{mainumb} imply in particular a rigidity result in the class of cylindrically symmetric hypersurfaces $M\subset \Hn$ for any $n\geq 1$.

\section{Definitions and preliminaries}\label{sec2}

In this section we collect some preliminary material that will be used in the rest of the paper. We shall recall some known notions for the study of smooth hypersurfaces in $\Hn$, and we refer the reader to \cite{DGN03, P, Ni, CHMY, DGN07, RiRo, HP, Ri, CL, BTV, CCHY18, BBCH} for several insights and different perspectives and approaches to the geometry of submanifolds in various subRiemannian settings.\\
Being $\la \cdot, \cdot \ra$ the metric defined in the Introduction (with induced norm $|\cdot|$), we denote by $\na$ the Levi-Civita connection associated to this metric. A direct computation shows that for any $i,j=1,\dots,n$ the following holds
\begin{equation}\label{eq: cov der}
\begin{split}
& \na_{X_i} X_j=0, \quad \na_{X_i} Y_j=2\delta_{ij}T, \quad  \na_{X_i} T=-2 Y_i, \\
& \na_{Y_i} X_j=-2\delta_{ij}T, \quad \na_{Y_i} Y_j=0, \quad  \na_{Y_i} T=2 X_i, \\
& \na_{T} X_i=-2 Y_i, \quad \na_{T} Y_i=2 X_i, \quad  \na_{T} T=0.
\end{split}
\end{equation}
For any smooth vector field $V$ in the horizontal distribution $\mathcal{H}$ we define 
$$
J(V):=-\frac{1}{2}\na_V T.
$$
In this way we have $J(X_i)=Y_i$ and $J(Y_i)=-X_i$ for all $i\in\{1,\ldots,n\}$. Moreover, for any $V,W \in \mathcal{H}$, one can easily see that the following relations hold
\begin{equation}\label{eq: relations J}
\begin{split}
&\la J(V),W\ra=-\la V,J(W)\ra,\quad \la J(V),J(W)\ra=\la V,W\ra,\\
&J(\na_VW)=\na_V(JW), \,\,\, \text{and} \,\,\, \la [V,W], T\ra = 4\la J(V),W\ra.
\end{split}
\end{equation}
For any smooth vector field $V$ in $\Hn$, we will use the notation $\mathcal{P}_H(V)$ to denote its horizontal projection, being $\mathcal{P}_H$ the orthogonal projection onto $\mathcal{H}$. A special role will be played by the horizontal part of the position vector, i.e.
$$\xi^H:=\mathcal{P}_H(\xi)=\sum_{j=1}^n x_jX_j + y_jY_j.$$
We can show that
\begin{equation}\label{eq cov der zeta xih}
\na_{Z}\xi^H=Z+2\la J(Z),\xi^H\ra  T\qquad\mbox{ for any }Z\in \mathcal{H}.
\end{equation}
To see \eqref{eq cov der zeta xih}, we just write $Z=\sum_{j=1}^n\left(\alpha_jX_j+\beta_jY_j\right)$ and it is straightforward to recognize from \eqref{eq: cov der} that
\begin{align*}
\na_{Z}\xi^H&=\sum_{j=1}^n\left(Z(x_j)X_j+Z(y_j)Y_j\right)+\\
&+\sum_{j,k=1}^n \left(\alpha_kx_j\na_{X_k}X_j+\beta_kx_j\na_{Y_k}X_j +\alpha_ky_j\na_{X_k}Y_j+\beta_ky_j\na_{Y_k}Y_j\right)\\
&=Z+2\sum_{j=1}^n(\alpha_jy_j-\beta_jx_j)T=Z+2\la J(Z), \xi^H\ra  T.
\end{align*}
Similarly we have
\begin{equation}\label{tderivative}
Z(t)=-2\left\langle J(Z),\xi^H \right\rangle \qquad\mbox{ for any }Z\in \mathcal{H},
\end{equation}
since in the same notations we can check that
$$
Z(t)=\sum_{j=1}^n \left(-2\alpha_j y_j + 2\beta_j x_j\right)=-2\left\langle J(Z),\xi^H \right\rangle.
$$
\vskip 0.4cm

We now start considering a $C^2$-smooth codimension $1$ submanifold $M$ in $\Hn$. We always assume $M$ to be connected and orientable. We denote by $\nu$ a fixed choice for the metric normal with unit length, and by $T_\xi M$ the tangent space at $\xi\in M$. Whenever $M$ is also compact and without boundary, we agree to fix $\nu$ as the outward unit normal. The characteristic set is defined as
$$S_M:=\{\xi\in M\,:\, \mathcal{P}_H(\nu)=0 \}=\{\xi\in M\,:\, T_\xi M=\mathcal{H}_\xi \}.$$
Outside of the set $S_M$ we suppose the hypersurface to be $C^\infty$-smooth. For any point in $M\smallsetminus S_M$ it is well-defined
$$
\nu^H=\frac{1}{|\mathcal{P}_H(\nu)|} \mathcal{P}_H(\nu),
$$
and we can write
$$
\nu= |\mathcal{P}_H(\nu)| \nu^H + \left\langle \nu, T\right\rangle T
$$
and define the tangent vector field
$$
\tau:= \left\langle \nu, T\right\rangle \nu^H - |\mathcal{P}_H(\nu)| T.
$$
We further denote
\begin{equation}\label{defeta}
\eta=-J\nu^H
\end{equation}
which clearly belongs to $\mathcal{H}\cap TM$. In case $n>1$, locally around any point $\xi\in M\smallsetminus S_M$ we can also pick smooth horizontal vector fields $V_i, W_i$ for $i=1,\ldots, n-1$ such that $J(V_i)=W_i$ and
$$
\left\{\eta, \nu^H, V_1, W_1,\ldots, V_{n-1}, W_{n-1}\right\}
$$
is an orthonormal basis for $\mathcal{H}_\xi$. With these choices we have fixed the orthonormal frames for $TM$ and $\mathcal{H}\cap TM$ (outside of characteristic points) as, respectively,
$$
\left\{\tau, \eta, V_1, W_1,\ldots, V_{n-1}, W_{n-1}\right\}\quad\mbox{ and }\quad\left\{\eta, V_1, W_1,\ldots, V_{n-1}, W_{n-1}\right\}.
$$
In our notations we have the following
\begin{lemma}
In $M\smallsetminus S_M$ it holds
\begin{equation}\label{deZnuhT}
\left\langle \na_Z \nu^H, T\right\rangle=2\left\langle \eta, Z\right\rangle \quad\mbox{ for every }Z \in \mathcal{H},
\end{equation}
\begin{equation}\label{partdivfree}
\left\langle \na_\tau \nu^H,\tau \right\rangle=0=\left\langle \na_\nu \nu^H,\nu \right\rangle,
\end{equation}
and
\begin{equation}\label{aggcommu}
\left\langle [Z_1,Z_2],\nu^H\right\rangle=\frac{-4\left\langle \nu , T\right\rangle}{|\mathcal{P}_H(\nu)|}\left\langle J(Z_1),Z_2\right\rangle\quad\mbox{ for every }Z_1,Z_2 \in \mathcal{H}\cap TM.
\end{equation}
\end{lemma}
\begin{proof}
The relation \eqref{deZnuhT} follows by \eqref{eq: relations J} and \eqref{defeta} since, for all $Z \in \mathcal{H}$, we have
$$
\left\langle \na_Z \nu^H, T\right\rangle=-\left\langle  \nu^H, \na_Z T\right\rangle = 2 \left\langle  \nu^H, J(Z)\right\rangle = 2 \left\langle  -J \nu^H, Z\right\rangle=2\left\langle \eta, Z\right\rangle.
$$
On the other hand, using that $|\nu^H|=1$ together with \eqref{eq: cov der}-\eqref{eq: relations J}, we obtain
\begin{align*}
&\left\langle \na_\nu \nu^H,\nu \right\rangle=\left\langle \na_\nu \nu^H,|\mathcal{P}_H(\nu)| \nu^H + \left\langle \nu, T\right\rangle T \right\rangle\\
&=\left\langle \nu, T\right\rangle \left\langle \na_\nu \nu^H,  T \right\rangle=- \left\langle \nu, T\right\rangle \left\langle  \nu^H,  \na_\nu T \right\rangle=- \left\langle \nu, T\right\rangle \left\langle  \nu^H,  \na_{|\mathcal{P}_H(\nu)| \nu^H} T \right\rangle \\
&=2|\mathcal{P}_H(\nu)| \left\langle \nu, T\right\rangle \left\langle  \nu^H,  J\nu^H \right\rangle=0
\end{align*}
and analogously
\begin{align*}
&\left\langle \na_\tau \nu^H,\tau \right\rangle =\left\langle \na_\tau \nu^H,\left\langle \nu, T\right\rangle \nu^H - |\mathcal{P}_H(\nu)| T\right\rangle\\
&=- |\mathcal{P}_H(\nu)| \left\langle \na_\tau \nu^H,  T \right\rangle=|\mathcal{P}_H(\nu)|  \left\langle  \nu^H,  \na_\tau T \right\rangle=|\mathcal{P}_H(\nu)| \left\langle  \nu^H,  \na_{\left\langle \nu, T\right\rangle \nu^H} T \right\rangle \\
&=-2|\mathcal{P}_H(\nu)| \left\langle \nu, T\right\rangle \left\langle  \nu^H,  J\nu^H \right\rangle=0.
\end{align*}
The previous two identities show \eqref{partdivfree}. Finally, in order to prove \eqref{aggcommu}, we can pick any $Z_1,Z_2 \in \mathcal{H}\cap TM$ and deduce from the property $[Z_1,Z_2]\in TM$ and \eqref{eq: relations J} that
\begin{align*}
&\left\langle [Z_1,Z_2],\nu^H\right\rangle = \frac{1}{|\mathcal{P}_H(\nu)|} \left\langle [Z_1,Z_2],|\mathcal{P}_H(\nu)|\nu^H\right\rangle \\
&= \frac{1}{|\mathcal{P}_H(\nu)|} \left\langle [Z_1,Z_2],\nu\right\rangle - \frac{ \left\langle \nu, T\right\rangle }{|\mathcal{P}_H(\nu)|} \left\langle [Z_1,Z_2],T\right\rangle = - \frac{ \left\langle \nu, T\right\rangle }{|\mathcal{P}_H(\nu)|} \left\langle [Z_1,Z_2],T\right\rangle\\
&= \frac{ -4\left\langle \nu, T\right\rangle }{|\mathcal{P}_H(\nu)|} \left\langle J(Z_1),Z_2\right\rangle.
\end{align*}
\end{proof}

We are then ready to recall the definition of horizontal mean curvature. Such notion arises in the criticality condition for the horizontal perimeter (see \cite{CDG, DGN07}).

\begin{definition}[horizontal mean curvature]\label{defcurv}
Let $M\subset \Hn$ as above. For any $\xi\in M\smallsetminus S_M$ we define the horizontal mean curvature of $M$ at $\xi$ as
\begin{equation}\label{defH}
H_M(\xi)=\frac{{\rm{div}}(\nu^H)}{2n-1}=\frac{1}{2n-1}\left( \left\langle \na_\eta \nu^H,\eta \right\rangle +\sum_{i=1}^{n-1} \left\langle \na_{V_i} \nu^H,V_i \right\rangle + \left\langle \na_{W_i} \nu^H,W_i \right\rangle \right),
\end{equation}
where $\rm{div}$ stands for the divergence with respect to the metric $\left\langle \cdot,\cdot \right\rangle$.\\
In particular, if $M\subset \mathbb{H}^1$ we simply have
\begin{equation}\label{defHn1}
H_M(\xi)=\left\langle \na_\eta \nu^H,\eta \right\rangle \quad\, \mbox{ in case }n=1.
\end{equation}
\end{definition}

We warn the reader that the second equality in \eqref{defH} is justified by \eqref{partdivfree}. The definition of $H_M$ can be (and, in the literature, has been) in fact given in multiple ways. For example, since $\na_T T=0$, it is immediate to recognize that
$$
{\rm{div}}(\nu^H)= \sum_{i=1}^{n} \left\langle \na_{X_i} \nu^H,X_i \right\rangle + \left\langle \na_{Y_i} \nu^H,Y_i \right\rangle.
$$
By noticing that by \eqref{deZnuhT} we have
$$
\mathcal{P}_H(\na_Z\nu^H)= \na_Z\nu^H - 2\left\langle \eta, Z\right\rangle T \quad\mbox{ for any }Z\in \mathcal{H},
$$
we can also recall the notion of horizontal shape operator (see \cite{Ri}) which we will be needed in what follows.

\begin{definition}[horizontal shape operator]\label{defshape}
Let $M\subset \Hn$ as above. For any $\xi\in M\smallsetminus S_M$ we can define the symmetric endomorphism $A_M(\cdot)(\xi)$ on $\mathcal{H}_\xi\cap T_\xi M$ as
$$
A_M(Z)= \mathcal{P}_H(\na_Z\nu^H) - \frac{2\left\langle \nu, T\right\rangle}{|\mathcal{P}_H(\nu)|}\left( J(Z)-\left\langle \eta, Z\right\rangle \nu^H \right)$$
for $Z\in \mathcal{H}\cap T M$.
\end{definition}
The fact that $A_M(Z) \in \mathcal{H}\cap T M$ for $Z\in \mathcal{H}\cap T M$ follows by the two identities
$$
\left\langle A_M(Z), T\right\rangle = 0 = \left\langle A_M(Z), \nu^H\right\rangle
$$
which can be easily checked. On the other hand, the symmetry of $A_M(\cdot)$ can be deduced from \eqref{aggcommu} since
\begin{align*}
&\left\langle A_M(Z_1), Z_2\right\rangle - \left\langle A_M(Z_2), Z_1\right\rangle \\
&=\left\langle \na_{Z_1}\nu^H, Z_2\right\rangle - \left\langle \na_{Z_2}\nu^H, Z_1\right\rangle - \frac{2\left\langle \nu, T\right\rangle}{|\mathcal{P}_H(\nu)|}\left(\left\langle J(Z_1),Z_2 \right\rangle - \left\langle J(Z_2),Z_1 \right\rangle\right)\\
&= - \left\langle \nu^H, \na_{Z_1} Z_2 - \na_{Z_2} Z_1\right\rangle - \frac{4\left\langle \nu, T\right\rangle}{|\mathcal{P}_H(\nu)|}\left\langle J(Z_1),Z_2 \right\rangle =0
\end{align*}
for any $Z_1, Z_2\in \mathcal{H}\cap T M$. When $n=1$, $\mathcal{H}\cap T M$ is $1$-dimensional (and generated by $\eta$ in our notations) and $A_M(\cdot)$ is nothing but the multiplication by the factor $H_M$. In higher dimensions the horizontal mean curvature appears as the normalized trace of $A_M$ since
\begin{align}\label{umbcurv}
&\left\langle A_M(\eta), \eta \right\rangle +\sum_{i=1}^{n-1} \left\langle A_M(V_i),V_i \right\rangle + \left\langle A_M(W_i),W_i \right\rangle \notag\\
&=\left\langle \na_\eta \nu^H,\eta \right\rangle +\sum_{i=1}^{n-1} \left\langle \na_{V_i} \nu^H,V_i \right\rangle + \left\langle \na_{W_i} \nu^H,W_i \right\rangle = (2n-1) H_M.
\end{align}
The following notion of horizontally umbilical hypersurface was introduced and studied in \cite{CCHY18, CCHY16}.

\begin{definition}\label{def: umbilical} Let $n\geq 2$. We say that $M$ is umbilic if, in $M\smallsetminus S_{M}$ it holds
\begin{align*}
A_M(Z)=(l-k)\la \eta,Z\ra\eta+kZ\qquad \forall\, Z\in \mathcal{H}\cap T M,
\end{align*}
for some suitable functions $k,l$.\\
In particular, at any non-characteristic point $\xi$, one has by \eqref{umbcurv}
$$
H_M(\xi)=\frac{1}{2n-1} ( l(\xi) + (2n-2) k(\xi) ).
$$
\end{definition}

The class of umbilic hypersurfaces is wide enough to contain any $M$ which is rotationally symmetric with respect to the vertical axis $\{(0,0,t)\in\Hn\,:\, t\in\R\}$ (see in this respect \cite[Proposition 3.1]{CCHY18}; see also the proof of Corollary \ref{corcyl} below).

\begin{remark}\label{rmk: umbilic cr} It is evident from Definition \ref{def: umbilical} that
\begin{equation}\label{l3k}
\mbox{if $M$ is umbilic with $l=3k$ then }\,k(\xi)= \frac{2n-1}{2n+1} H_M(\xi)\mbox{ for all }\xi\in M\smallsetminus S_M.
\end{equation}
The case $l=3k$ is related to the gauge spheres (see Example \ref{exsphere} below). We also mention the case $l=2k$ which is the one studied in \cite{CCHY18} and it is strictly related to the Pansu spheres.
\end{remark}

Let us compute explicitly the objects previously discussed in the particular case of the gauge spheres.

\begin{example}\label{exsphere}
Let $n\geq 1$, $\xi_0=(0,0,t_0)\in \Hn$ and $M=\partial B_R(\xi_0)=\{\xi=(x,y,t)\,:\, (|x|^2+|y|^2)^2+(t-t_0)^2=R^4\} \subset \Hn$. We use the notation $r=r(x,y)=\sqrt{|x|^2+|y|^2}=|\xi^H|$ to denote the distance from the $t$-axis. For any $\xi\in M$ one has
\begin{equation}\label{componu}
|\mathcal{P}_H(\nu)|=\frac{2rR^2}{\sqrt{4r^2R^4+(t-t_0)^2}}\quad\mbox{ and }\quad \left\langle \nu,T \right\rangle=\frac{t-t_0}{\sqrt{4r^2R^4+(t-t_0)^2}},
\end{equation}
which is saying in particular that the characteristic points are the ones places on the $t$-axis, i.e. $S_M=\{(0,0,t_0\pm R^2)\}$. Outside of these two points we have
\begin{align}\label{eq: coeff of nuH}
\nu^H&=\sum_{j=1}^n \frac{r^2x_j-y_j(t-t_0)}{rR^2} X_j+\frac{r^2y_j+x_j(t-t_0)}{rR^2} Y_j=\frac{r^2\xi^H+(t-t_0)J\xi^H}{rR^2}\\
\mbox{and }\quad \eta&=-J\nu^H=\frac{(t-t_0)\xi^H-r^2J\xi^H}{rR^2}.\notag
\end{align}
A straightforward computation then shows for any $j,k\in\{1,\ldots,n\}$
\begin{align*}
&\la\na_{X_k}\nu^H,X_j\ra=X_k(\left\langle \nu^H,X_j\right\rangle)
=\frac{1}{r^2R^2}\left((2x_jx_k+\delta_{jk}r^2+2y_jy_k)r-\left\langle \nu^H,X_j\right\rangle R^2x_k\right),\\
&\la\na_{Y_k}\nu^H,X_j\ra=Y_k(\left\langle \nu^H,X_j\right\rangle)
=\frac{1}{r^2R^2}\left((2x_jy_k-\delta_{jk}(t-t_0)-2y_jx_k)r-\left\langle \nu^H,X_j\right\rangle R^2y_k\right),\\
&\la\na_{X_k}\nu^H,Y_j\ra=X_k(\left\langle \nu^H,Y_j\right\rangle)
=\frac{1}{r^2R^2}\left((2y_jx_k+\delta_{jk}(t-t_0)-2x_jy_k)r-\left\langle \nu^H,Y_j\right\rangle R^2x_k\right),\\
&\la\na_{Y_k}\nu^H,Y_j\ra=Y_k(\left\langle \nu^H,Y_j\right\rangle)
=\frac{1}{r^2R^2}\left((2y_jy_k+\delta_{jk}r^2+2x_jx_k)r-\left\langle \nu^H,Y_j\right\rangle R^2y_k\right),
\end{align*}
which we can rewrite using \eqref{eq: coeff of nuH} in the following way
\begin{align}\label{pnuhsfera}
\mathcal{P}_H(\na_Z\nu^H)&=\frac{1}{R^2}\left(\frac{2}{r}\left(\la\xi^H,Z\ra\xi^H+\la J\xi^H,Z\ra J\xi^H\right)+rZ+\frac{t-t_0}{r}JZ-\frac{R^2}{r^2}\la\xi^H,Z\ra\nu^H\right)\notag\\
&=\frac{1}{R^2}\left(2r\la \eta,Z\ra\eta+rZ+\frac{t-t_0}{r}JZ-\frac{t-t_0}{r}\la\eta,Z\ra\nu^H\right)
\end{align}
for any horizontal vector $Z$. It is then easy to check that
\begin{equation}\label{curvsphere}
H_M(\xi)=\frac{2n+1}{2n-1} \frac{r}{R^2}.
\end{equation}
Also, recalling Definition \ref{defshape} and using \eqref{componu} and \eqref{pnuhsfera}, we can recognize 
$$
A_M(Z)= \frac{2r}{R^2}\la \eta,Z\ra\eta+ \frac{r}{R^2}Z\qquad \forall\, Z\in \mathcal{H}\cap T M.
$$
According to Definition \ref{def: umbilical}, when $n\geq 2$ this is saying that $M$ is umbilic with $l(\xi)=3 k(\xi)$ and $k(\xi)=\frac{|\xi^H|}{R^2}$.
\end{example}

It is well known in the literature that, whenever $n\geq 2$, the horizontal and tangent vector fields in $\mathcal{H}\cap TM$ satisfy an H\"ormander type property as they can reproduce any tangent direction via commutation. If $M$ is also umbilic such information can be made very precise and it is encoded in the following lemma.

\begin{lemma}\label{commutator} Let $n\geq 2$. For $\xi\in M\smallsetminus S_M$ denote $$\mathcal{H}^0_\xi={\rm{span}}\{V_1, W_1,\ldots, V_{n-1}, W_{n-1}\}.$$
If $M$ is umbilical then
\begin{align*}
&{\rm{span}}\{Z, [Z_1, Z_2]\,:\, \mbox{with }Z, Z_1, Z_2 \in \mathcal{H}^0_\xi\}=\\
&={\rm{span}}\left\{ V_1, W_1,\ldots, V_{n-1}, W_{n-1}, \tau-\frac{k(\xi)|\mathcal{P}_H(\nu)|}{2}\,\eta\right\}.
\end{align*}
\end{lemma}
\begin{proof}
Fix any $Z_1, Z_2 \in \mathcal{H}^0_\xi$. By \eqref{eq: relations J} and \eqref{aggcommu} we have
\begin{align*}
&\left\langle [Z_1,Z_2], \tau\right\rangle=\left\langle \nu, T\right\rangle \left\langle [Z_1,Z_2], \nu^H\right\rangle - |\mathcal{P}_H \nu| \left\langle [Z_1,Z_2], \nu^H\right\rangle \\
&=-4\left(\frac{\left\langle \nu, T\right\rangle^2}{|\mathcal{P}_H \nu|} +|\mathcal{P}_H \nu|\right)\left\langle J(Z_1),Z_2\right\rangle =\frac{-4}{|\mathcal{P}_H \nu|}\left\langle J(Z_1),Z_2\right\rangle,
\end{align*}
which says that ${\rm{span}}\{Z, [Z_1, Z_2]\,:\, \mbox{with }Z, Z_1, Z_2 \in \mathcal{H}^0_\xi\}$ is at least ($2n-1$)-dimensional. On the other hand, using also \eqref{defeta} and the commutation property of $J$ and $\nabla$ together with the umbilicality of $M$, we obtain
\begin{align*}
&\left\langle [Z_1,Z_2], \eta\right\rangle=\left\langle \nabla_{Z_1}JZ_2-\nabla_{Z_2}JZ_1, \nu^H\right\rangle = \left\langle J(Z_1),\nabla_{Z_2}\nu^H \right\rangle - \left\langle J(Z_2),\nabla_{Z_1}\nu^H \right\rangle\\
&= \left\langle J(Z_1),A_M\left(Z_2\right)+ \frac{2\left\langle \nu, T\right\rangle}{|\mathcal{P}_H \nu|}J(Z_2) \right\rangle - \left\langle J(Z_2),A_M\left(Z_1\right)+ \frac{2\left\langle \nu, T\right\rangle}{|\mathcal{P}_H \nu|}J(Z_1) \right\rangle\\
&=2k\left\langle J(Z_1),Z_2\right\rangle.
\end{align*}
Hence we get
$$\left\langle [Z_1,Z_2], \eta+ \frac{k(\xi)|\mathcal{P}_H(\nu)|}{2}\tau\right\rangle = 0\qquad\mbox{for every } Z_1 , Z_2 \in \mathcal{H}^0_\xi.
$$
This implies that ${\rm{span}}\{Z, [Z_1, Z_2]\,:\, \mbox{with }Z , Z_1 , Z_2\in \mathcal{H}^0_\xi\}$ is exactly ($2n-1$)-dimensional and the vector $\tau-\frac{k(\xi)|\mathcal{P}_H(\nu)|}{2}\,\eta$ belongs to such vector space as desired.
\end{proof}

\section{Darboux-type results}\label{sec3}

\subsection{The case of $\mathbb{H}^1$}

In this section we first treat the ($n=1$)-dimensional case by providing the proof of Theorem \ref{main}. As we mentioned in the introduction and recalled in \eqref{defHn1}, for surfaces $M$ in $\mathbb{H}^1$ the main role is played by the integral curves of the only horizontal and tangent vector field $\eta$. A (naive) way to describe our approach to Theorem \ref{main} is to draw a parallelism with the classical problem of identifying the pieces of circles as the only smooth connected curves $\Gamma$ in $\R^2$ with non-zero constant curvature $K=K_\Gamma$. Among the many ways to show this property, a very direct one is to consider (denoting with $p=(p_1,p_2)$ the generic point in $\R^2$ and with $N$ a choice for the unit normal to $\Gamma$) the two functions
\begin{equation}\label{twof}
\begin{cases}
f_1(p)=K p_1- \la N,\partial_{p_1}\ra, \qquad\quad f_1: \Gamma \to \R,\\
f_2(p)=K p_2- \la N,\partial_{p_2}\ra, \qquad\quad f_2: \Gamma \to \R.
\end{cases}
\end{equation}
By differentiating along a unit tangent vector $U$ and using $K=\left\langle \nabla_U N,U\right\rangle$, one recognizes that $U f_1=Uf_2= 0$ on $\Gamma$. Thus, there have to exist two constants $c_1,c_2$ such that $f_i\equiv c_i$, $i=1,2$, and we have $1=\la N,\partial_{p_1}\ra^2+\la N,\partial_{p_2}\ra^2=\left(K p_1 - c_1\right)^2+\left(K p_2 - c_2\right)^2$ for $p\in\Gamma$, i.e.
$\Gamma$ is contained in the circle of radius $\frac{1}{|K|}$ and center $(\frac{c_1}{K},\frac{c_2}{K})$. If we bring back the attention to the case of the 2-dimensional surface $M$ in $\mathbb{H}^1$, we emphasize that in Theorem \ref{main} we prescribe the curvature $H_M(\xi)$ to be proportional to $|\xi^H|$ (see also \eqref{curvsphere} in Example \ref{exsphere}). The term $|\xi^H|$ corresponds to the distance (either Euclidean distance or gauge-related distance, as they coincide in this case) to the vertical line $L_v$ defined by
$$L_v=\{(0,0,t)\in \mathbb{H}^1\,:\, t\in\R\}.$$ 
Having this in mind, as well as the notations introduced in Section \ref{sec2}, we define the two functions
\begin{equation}\label{defphis}
\begin{cases}
\varphi_h(\xi)=\frac{1}{3}H_M(\xi)|\xi^H|^2 - \la \nu^H,\xi^H\ra, \qquad\quad \varphi_h: M\smallsetminus S_M \to \R,\\
\varphi_v(\xi)=\frac{1}{3}H_M(\xi) \frac{t}{|\xi^H|} - \la \eta,\frac{\xi^H}{|\xi^H|}\ra\qquad\quad\,\,\,\,\,\,\,\, \varphi_v: M\smallsetminus \left(S_M\cup L_v\right) \to \R.
\end{cases}
\end{equation}
The functions $\varphi_h$ and $\varphi_v$ have a different role. In the following lemma we show that $\varphi_h$ is in fact constant along the integral curves of $\eta$, whereas the behaviour of $\varphi_v$ is subordinate to the one of $\varphi_h$.

\begin{lemma}\label{lem1}
Let $M$ be a smooth surface in $\H^1$ which is connected and orientable. Let also $\omega$ be a relatively open set contained in $M\smallsetminus S_M$. Suppose there exists $c\in\R$ such that $H_M(\xi)=c|\xi^H|$ for $\xi\in\omega $. Then we have
$$\begin{cases}
\eta(\varphi_h)=0\qquad\qquad\quad\,\,\,\, \mbox{ in }\omega,\\
\eta(\varphi_v)=\frac{\la \nu^H,\xi^H\ra}{|\xi^H|^3}\varphi_h \qquad \mbox{ in }\omega\smallsetminus L_v.
\end{cases}$$
\end{lemma}
\begin{proof}
By \eqref{eq cov der zeta xih} and \eqref{defHn1}, we have
\begin{equation}\label{dernuhxi}
\eta(|\xi^H|)=\frac{\left\langle \eta,\xi^H\right\rangle}{|\xi^H|} \quad\mbox{ and }\quad \eta(\la \nu^H,\xi^H\ra)=H_M \left\langle \eta,\xi^H\right\rangle,
\end{equation}
where in the second equality we also exploited the fact that $\mathcal{H}_\xi$ is generated by the two orthogonal unit vectors $\eta$ and $\nu^H$. Hence, for $\xi \in \omega$, we obtain
$$
\eta(\varphi_h)=\eta\left(\frac{c}{3}|\xi^H|^3-\la \nu^H,\xi^H\ra\right)=c|\xi^H|^2\frac{\left\langle \eta,\xi^H\right\rangle}{|\xi^H|} - H_M(\xi) \left\langle \eta,\xi^H\right\rangle =0.
$$
On the other hand, by \eqref{tderivative}-\eqref{defeta} we have 
\begin{equation}\label{etat}
\eta(t)=-2\left\langle \nu^H,\xi^H\right\rangle
\end{equation}
and, using also \eqref{eq cov der zeta xih}-\eqref{eq: relations J},
\begin{equation}\label{deretahxi}
\eta(\la \eta,\xi^H\ra)=1+\left\langle \nabla_\eta \eta, \xi^H\right\rangle= 1+\left\langle \nabla_\eta \nu^H, J\xi^H\right\rangle =1-H_M \left\langle \nu^H,\xi^H\right\rangle.
\end{equation}
Recalling that $|\xi^H|^2=\la \eta,\xi^H\ra^2+\la \nu^H,\xi^H\ra^2$, we then infer
\begin{align*}
&\eta(\varphi_v)=\eta\left(\frac{c}{3}t-\frac{\la \eta,\xi^H\ra}{|\xi^H|}\right)\\
&=\frac{-2c}{3}\left\langle \nu^H,\xi^H\right\rangle-\frac{1}{|\xi^H|}+H_M(\xi) \left\langle \nu^H,\frac{\xi^H}{|\xi^H|}\right\rangle + \frac{1}{|\xi^H|^3}\la \eta,\xi^H\ra^2\\
&=\left\langle \nu^H,\frac{\xi^H}{|\xi^H|}\right\rangle\left(\frac{-2c}{3}|\xi^H| + H_M(\xi)\right)-\frac{1}{|\xi^H|^3}\left(|\xi^H|^2-\la \eta,\xi^H\ra^2\right)\\
&=\frac{\left\langle \nu^H,\xi^H\right\rangle}{|\xi^H|^3}\left(\frac{1}{3}H_M(\xi)|\xi^H|^2-\left\langle \nu^H,\xi^H\right\rangle\right)=\frac{\left\langle \nu^H,\xi^H\right\rangle}{|\xi^H|^3}\varphi_h(\xi)
\end{align*}
whenever $\xi^H\neq 0$. This completes the proof of the lemma.
\end{proof}

Keeping in mind the comparison between the derivatives of \eqref{defphis} along $\eta$ in Lemma \ref{lem1} and the derivatives along the curve $\Gamma$ of \eqref{twof}, it is no surprise that we want $\varphi_h$ to vanish identically throughout $M$. This is exactly what we show in the next lemma. We will deduce this fact from the global properties of the integral curves of $\eta$ and from the assumption $S_M\subset L_v$.

\begin{lemma}\label{lem2}
Let $M$ be a smooth surface in $\H^1$ which is connected, orientable, compact, and without boundary. Assume that $S_M\subseteq M\cap L_v$, and that there exists $c\neq 0$ such that $H_M(\xi)=c|\xi^H|$ for every point $\xi\in M\smallsetminus S_M$. Then $c>0$, $\varphi_h\equiv 0$, and every integral curve of $\eta$ reaches $S_M$.
\end{lemma}
\begin{proof}
Let us divide the proof in three steps.\\
\noindent{\it Step I. } In the first step we shall show that $c>0$. By the compactness of $M$ the function $\frac{1}{2}|\xi^H|^2$ attains its maximum at a point $\xi_1\in M\smallsetminus L_v$. Since we have $S_M\subseteq M\cap L_v$, at $\xi=\xi_1$ we have
$$
0=\eta\left(\frac{1}{2}|\xi^H|^2\right)=\left\langle \eta, \xi_1^H\right\rangle\quad\mbox{and}\quad 0=\tau\left(\frac{1}{2}|\xi^H|^2\right)=\left\langle \nu,T\right\rangle \left\langle \nu^H, \xi_1^H\right\rangle.
$$ 
Since $\la \nu^H,\xi_1^H\ra^2=\la \eta,\xi_1^H\ra^2+\la \nu^H,\xi_1^H\ra^2= |\xi_1^H|^2>0$, we have that $\left\langle \nu,T\right\rangle=0$ and therefore
$$
\left\langle \nu^H, \xi_1^H\right\rangle = \frac{1}{|\mathcal{P}_H \nu|} \left\langle \nu, \xi_1 \right\rangle >0
$$
where the positive sign is a consequence of the maximality condition and the fact that $\nu$ is the outward normal. Hence, by \eqref{deretahxi} at the maximum point $\xi=\xi_1$ we obtain
$$
1-H_M(\xi_1) \left\langle \nu^H, \xi_1^H\right\rangle= \eta^2\left(\frac{1}{2}|\xi^H|^2\right)\leq 0
$$
which says that
$$
c\geq \frac{1}{|\xi_1^H|\left\langle \nu^H, \xi_1^H\right\rangle}>0.
$$
\noindent{\it Step II. } We now prove that
$$
\varphi_h\equiv 0.
$$ 
By contradiction we shall assume the existence of $\xi_0\in M\smallsetminus S_M$ such that $\varphi_h(\xi_0)\neq 0$. Since by definition we have
\begin{equation}\label{realphih}
\varphi_h(\xi)=\frac{c}{3}|\xi^H|^3-\left\langle \nu^H,\xi^H\right\rangle,
\end{equation}
it is clear that $\varphi_h$ vanishes on the vertical line $L_v$. Therefore we know that $\xi_0 \in M\smallsetminus L_v$. Let us consider the integral curve $\gamma$ of $\eta$ starting from $\xi_0$. Lemma \ref{lem1} implies that $\varphi_h$ is constant along $\gamma$, i.e.
$$
\varphi_h(\gamma(s))=\varphi_h(\xi_0)=:\varphi_0.
$$
Since $S_M\subset L_v$ and $\varphi_0\neq 0$, $\gamma$ remains in $M\smallsetminus L_v$ and there is no problem in extending the curve indefinitely. We claim that this fact will contradict the boundedness of $M\supset \gamma$. Denote by $t(s)$, $r(s)$, and $\theta(s)$ the three smooth functions defined for $\xi\in \gamma$ respectively by
\begin{align*}
&t(s)=\gamma_3(s),\qquad r(s)=\left(\gamma^2_1(s)+\gamma^2_2(s)\right)^{\frac{1}{2}}=|\xi^H|,\quad\mbox{ and }\\
&\begin{cases}
\cos{\left(\theta(s)\right)} = \left\langle \nu^H,\frac{\xi^H}{|\xi^H|}\right\rangle, \\
\sin{\left(\theta(s)\right)} = \left\langle \eta,\frac{\xi^H}{|\xi^H|}\right\rangle.
\end{cases}
\end{align*}
From \eqref{realphih} we readily recognize 
$$
\frac{c}{3}r^3(s)-r(s)\cos{\left(\theta(s)\right)}=\varphi_0,
$$
which implies that along the curve $\gamma$ the positive function $r(s)$ is in fact a function of $\cos{\left(\theta(s)\right)}$ (in the sense that it is uniquely determined by the value $\cos{\left(\theta(s)\right)}$). As we will make use of this fact, we set the notation $R(\cos(\theta(s)))=r(s)$. From \eqref{etat} we have that
\begin{equation}\label{tprime}
t'(s)=-2r(s)\cos{\left(\theta(s)\right)}=2\varphi_0-\frac{2c}{3}r^3(s).
\end{equation}
Thus, if $\varphi_0<0$ then $t'(s)\leq 2\varphi_0<0$ and $t(s)$ would be forced to be unbounded providing an immediate contradiction. We can then assume $\varphi_0>0$. Since from \eqref{dernuhxi} and \eqref{deretahxi} we have
\begin{align*}
\eta\left(\arctan\left(\frac{\left\langle \eta,\xi^H\right\rangle}{\left\langle \nu^H,\xi^H\right\rangle}\right)\right)&=\frac{(1-H_M\left\langle \nu^H,\xi^H\right\rangle)\left\langle \nu^H,\xi^H\right\rangle - H_M \left\langle \eta,\xi^H\right\rangle^2 }{\left\langle \eta,\xi^H\right\rangle^2+\left\langle \nu^H,\xi^H\right\rangle^2}\\
&=\frac{\left\langle \nu^H,\xi^H\right\rangle-H_M |\xi^H|^2}{|\xi^H|^2}= \frac{-\frac{2c}{3} |\xi^H|^3 - \varphi_h(\xi)}{|\xi^H|^2},
\end{align*}
we obtain
\begin{equation}\label{thetaprime}
\theta'(s)=\frac{-\frac{2c}{3} r^3(s) - \varphi_0}{r^2(s)}.
\end{equation}
The assumption $\varphi_0>0$ implies that $\theta(s)$ is strictly decreasing, so that the angle formed by (the horizontal projections of) $\nu^H$ and $\xi^H$ attains every value in $[0,2\pi]$ infinitely many times along $\gamma$. We can then consider a strictly increasing sequence of values $\{s_k\}_{k\in\N}$ such that $\theta(s_k)-\theta(s_{k+1})=2\pi$ for all $k\in\N$. By exploiting \eqref{tprime} and \eqref{thetaprime} we notice that
\begin{align*}
&t(s_{k+1})-t(s_k)=\int_{s_k}^{s_{k+1}} t'(s) ds \\
&= -2 \int_{s_k}^{s_{k+1}} r(s)\cos(\theta(s)) ds=2 \int_{s_k}^{s_{k+1}} \frac{r^3(s)\cos(\theta(s))}{\frac{2c}{3} r^3(s) + \varphi_0} \theta'(s) ds\\
&= \frac{2}{c} \int_{s_k}^{s_{k+1}} \frac{(cr^3(s) - r(s)\cos(\theta(s)) + r(s)\cos(\theta(s)))\cos(\theta(s))}{\frac{2c}{3} r^3(s) + \varphi_0} \theta'(s) ds\\
&=  \frac{2}{c} \int_{s_k}^{s_{k+1}}\cos(\theta(s))\theta'(s) ds + \frac{2}{c} \int_{s_{k}}^{s_{k+1}}  \frac{r(s)\cos^2(\theta(s))}{\frac{2c}{3} r^3(s) + \varphi_0} \theta'(s) ds\\
&=\frac{2}{c} \int_{s_k}^{s_{k+1}}  \frac{r(s)\cos^2(\theta(s))}{\frac{2c}{3} r^3(s) + \varphi_0} \theta'(s) ds\\
&=-\frac{2}{c} \int_{\theta(s_{k+1})}^{\theta(s_{k})}  \frac{R(\cos(\sigma))\cos^2(\sigma)}{\frac{2c}{3} R^3(\cos(\sigma)) + \varphi_0} d\sigma=-\frac{2}{c} \int_{\theta(s_{1})-2\pi}^{\theta(s_{1})}  \frac{R(\cos(\sigma))\cos^2(\sigma)}{\frac{2c}{3} R^3(\cos(\sigma)) + \varphi_0} d\sigma
\end{align*}
for every $k\in\N$. This implies, also in the case $\varphi_0>0$, the unboundedness of $t(s)$ since
$$
t(s_{k+1})=t(s_1)-k\frac{2}{c}\int_{\theta(s_{1})-2\pi}^{\theta(s_{1})}  \frac{R(\cos(\sigma))\cos^2(\sigma)}{\frac{2c}{3} R^3(\cos(\sigma)) + \varphi_0} d\sigma \to -\infty \quad\mbox{ as }k\to\infty.
$$
Therefore, under both the assumptions $\varphi_0<0$ and $\varphi_0>0$, we have reached a contradiction. This completes the proof of the identity $\varphi_h\equiv 0$.

\noindent{\it Step III. } We finally show that 
$$
S_M\not=\emptyset \mbox{ and every integral curve $\gamma$ of $\eta$ starting from any $\xi_0\in M\smallsetminus S_M$ reaches $S_M$.}
$$
Let us exploit the same notations of {{\it Step II}}. Arguing again by contradiction, we can assume that the curve $\gamma$ can be extended indefinitely. We stress that $r(s)$ can vanish (at points in $L_v\smallsetminus S_M$) but only at isolated points on the curve since $\eta_{\bar{\xi}}$ belongs to ${\rm{span}}\{\de_x, \de_y\}$ at points $\bar{\xi}\in L_v\smallsetminus S_M$. Also, the functions $r(s)$ and $\theta(s)$ are smooth outside $L_v$. Since $S_M\subseteq M\cap L_v$, two situations might occur: either there exists $s_0$ such that $\inf_{s\in (s_0,\infty)} r(s)>0$ or there exists a strictly increasing sequence of values $\{s_k\}_{k\in\N}$ such that, for every $k\in\N$, $r(s_k)=0$ and $r(s)>0$ for $s\in (s_k, s_{k+1})$. Since by {{\it Step II}} and \eqref{tprime} we have
\begin{equation}\label{scendee}
t'(s)=-\frac{2c}{3}r^3(s),
\end{equation}
we deduce that the occurrence of the first case leads to an immediate contradiction since $t'(s)\leq -\frac{2c}{3}\left( \inf_s r(s) \right)^3<0$ for $s>s_0$ and $t(s)$ would be unbounded. Hence, we can assume the existence of the sequence $\{s_k\}_{k\in\N}$ satisfying the above assumptions. Fix any $k\in \N$ and consider $s\in (s_k,s_{k+1})$. Using {{\it Step II}}, \eqref{thetaprime}, and {{\it Step I}}, we have
\begin{equation}\label{arrivaa}
\cos{\left(\theta(s)\right)}=\frac{c}{3}r^2(s)>0\qquad\mbox{ and }\qquad \theta'(s)=-\frac{2c}{3}r(s)<0.
\end{equation}
This yields
$$
\begin{cases}
(\cos{\left(\theta(s)\right)}, \sin{\left(\theta(s)\right)})\to (0,+1) \quad\mbox{ as }s\to s_k^+ \\
(\cos{\left(\theta(s)\right)}, \sin{\left(\theta(s)\right)})\to (0,-1) \quad\mbox{ as }s\to s_{k+1}^-. 
\end{cases}
$$ 
Hence we infer
\begin{align*}
&t(s_{k+1})-t(s_k)=\int_{s_k}^{s_{k+1}} t'(s) ds = -2 \int_{s_k}^{s_{k+1}} r(s)\cos(\theta(s)) ds\\
&=\frac{3}{c} \int_{s_k}^{s_{k+1}} \theta'(s) \cos(\theta(s)) ds =\frac{-6}{c}.
\end{align*}
In other words, each time the curve $\gamma$ re-joins the vertical line $L_v$ the $t$-component of the curve drops by a fixed amount. Since we are assuming that $\gamma$ is reaching $L_v$ an infinite number of times, this fact is in contradiction with the compactness of $M$. The proof is then complete.
\end{proof}

We are now ready to complete the proof of Theorem \ref{main}.

\begin{proof}[Proof of Theorem \ref{main}]
We start by noticing that, under our assumptions, the set $S_M$ (which is non-empty by Lemma \ref{lem2}) consists of isolated points. As a matter of fact, since $S_M\subset L_v$, if we had a sequence of points in $S_M$ converging to $\bar{\xi}\in S_M$ then such sequence would be in $L_v$ and at the point $\bar{\xi}$ the vector field $T$ would be tangent. On the other hand, the tangent space at the characteristic points coincides with the horizontal distribution which is ${\rm{span}}\{\de_x, \de_y\}$ on $L_v$. This argument ensures the fact that the characteristic points are isolated. Therefore, there exist $t_1<t_2<\ldots<t_p$ for some finite $p\in\N$ such that
$$
S_M=\{(0,0,t_1),\ldots,(0,0,t_p)\}.
$$ 
Consider now any point $\xi_0 \in M\cap \{t<t_2\}$ such that $\xi_0\not =(0,0,t_1)$, and consider the integral curve $\gamma$ of $\eta$ starting from $\xi_0$. Using the same notations as in Lemma \ref{lem2} we know that $t(s)$ is decreasing (see \eqref{scendee}) so that $\gamma \subset M\cap \{t<t_2\}$. Exploiting Lemma \ref{lem1} together with the identity $\varphi_h\equiv 0$ showed in Lemma \ref{lem2}, we have that the function $\varphi_v(\xi)=\frac{c}{3}t-\left\langle \eta,\frac{\xi^H}{|\xi^H|}\right\rangle$ is constant along $\gamma$, i.e.
$$
\frac{c}{3}t(s)-\sin(\theta(s))=\varphi_v(\xi_0).
$$
We stress that the previous identity holds true in $\gamma \smallsetminus L_v$, and it can then be extended by continuity on the whole $\gamma$. By Lemma \ref{lem2} we have that $\gamma$ reaches $S_M$, and in particular as $\gamma(s)\to (0,0,t_1)$ we have
$$
t(s)\to t_1\qquad \mbox{ and }\qquad\sin(\theta(s))\to -1
$$
(see also \eqref{arrivaa} in this respect). Hence
$$
\varphi_v(\xi_0)=\frac{c}{3}t_1+1.
$$
By the arbitrariness of $\xi_0 \in (M\cap \{t<t_2\})\smallsetminus\{(0,0,t_1)\}$ we have
$$
\varphi_v(\xi) = \frac{c}{3}t_1+1 \quad\mbox{ for all }\xi\in  (M\cap \{t<t_2\})\smallsetminus\{(0,0,t_1)\}.
$$
The two identities $\varphi_h\equiv 0$ and $\varphi_v \equiv \frac{c}{3}t_1+1$ can be rewritten as
$$
\left\langle \nu^H,\frac{\xi^H}{|\xi^H|}\right\rangle=\frac{c}{3}|\xi^H|^2\quad\mbox{ and }\quad \left\langle \eta,\frac{\xi^H}{|\xi^H|}\right\rangle=\frac{c}{3}\left(t-t_1-\frac{3}{c}\right),
$$
which implies
$$
1=\left\langle \nu^H,\frac{\xi^H}{|\xi^H|}\right\rangle^2+\left\langle \eta,\frac{\xi^H}{|\xi^H|}\right\rangle^2= \left(\frac{c}{3}|\xi^H|^2\right)^2+\left(\frac{c}{3}\left(t-t_1-\frac{3}{c}\right)\right)^2
$$
for any $\xi \in (M\cap \{t<t_2\})\smallsetminus\{(0,0,t_1)\}$. By the very definition of gauge sphere, this shows that
$$
(M\cap \{t<t_2\})\smallsetminus\{(0,0,t_1)\} \subset \de B_R(0,0,t_0)
$$
where
$$
R^2=\frac{3}{c}\qquad\mbox{ and }\qquad t_0=t_1+\frac{3}{c}.
$$
Being $M$ a smooth connected surface with no boundary and having the gauge sphere only two characteristic points at $(0,0,t_0-R^2)=(0,0,t_1)$ and $(0,0,t_0+R^2)=(0,0,t_1+\frac{6}{c})$, we can conclude that $M=\de B_R(0,0,t_0)$ as desired.  
\end{proof}

\subsection{The case of $\mathbb{H}^n$, $n\geq 2$}

Let us now turn the attention to the case $n\geq 2$ and to the proof of Theorem \ref{mainumb}. By pushing further the parallelism with the classical Euclidean framework, we can say that the higher dimensional analogue of the planar argument sketched in \eqref{twof} is effective if one requires the hypersurface to be (locally) umbilical: it provides in fact a proof of the classical characterization of umbilical surfaces also known as Darboux theorem \cite{D}  (see \cite{MoRos} for an expository text; see also \cite{MTsigma, GM} for different but related settings). In our Theorem \ref{mainumb} the main assumption is the umbilicality of $M$ with respect to Definition \ref{def: umbilical}. We warn the reader that in Definition \ref{def: umbilical} there is no information about the relationship between the two functions $l$ and $k$, and therefore a characterization is possible only under a prescription of the curvature (see in this respect \cite{CCHY18} for the case of constant $\sigma_k$-curvatures). Having this is mind, together with the fact that we are prescribing $H_M(\xi)=c|\xi^H|$, our aim is to provide a Darboux-type approach to Theorem \ref{mainumb}. We define
\begin{equation}\label{defphisn}
\begin{cases}
\varphi_h(\xi)=\frac{2n-1}{2n+1}H_M(\xi)|\xi^H|^2 - \la \nu^H,\xi^H\ra, \qquad\quad \varphi_h: M\smallsetminus S_M \to \R,\\
\varphi_v(\xi)=\frac{2n-1}{2n+1}H_M(\xi) \frac{t}{|\xi^H|} - \la \eta,\frac{\xi^H}{|\xi^H|}\ra\qquad\quad\,\,\,\,\,\,\,\, \varphi_v: M\smallsetminus \left(S_M\cup L_v\right) \to \R,
\end{cases}
\end{equation}
where we have kept the notation 
$$L_v=\{(0,0,t)\in \Hn\,:\, t\in\R\}$$ 
to denote the $t$-axis. With the following lemma we realize that the constancy of the key function $\varphi_h$ along $\eta$ is tied to the vanishing of $l-3k$ (in Example \ref{exsphere} we saw that for the gauge spheres $l=3k$ by a direct computation).

\begin{lemma}\label{lem1n}
Fix $n\geq 2$. Let $M$ be a smooth hypersurface in $\H^n$ which is connected and orientable. Let also $\omega$ be a relatively open set contained in $M\smallsetminus S_M$. Suppose there exists $c\in\R$ such that $H_M(\xi)=c|\xi^H|$ for $\xi\in\omega $. If $M$ is umbilic then we have
$$
\eta(\varphi_h)=\frac{2n-2}{2n+1}\left\langle \eta,\xi^H\right\rangle \left( 3k-l\right)\qquad\quad \mbox{ in }\omega.
$$
\end{lemma}
\begin{proof}
The umbilicality condition in Definition \ref{def: umbilical} implies that
\begin{equation}\label{umbder}
\la \nabla_\eta \nu^H,\xi^H\ra= l(\xi) \la \eta,\xi^H\ra.
\end{equation}
For $\xi\in\omega$ we can then exploit the assumption $H_M(\xi)=c|\xi^H|$, together with \eqref{umbder} and \eqref{eq cov der zeta xih}, to deduce that
\begin{align*}
\eta(\varphi_h)&=\eta\left(c\frac{2n-1}{2n+1}|\xi^H|^3-\la \nu^H,\xi^H\ra\right)=3c\frac{2n-1}{2n+1}|\xi^H|^2\frac{\left\langle \eta,\xi^H\right\rangle}{|\xi^H|} - l(\xi) \left\langle \eta,\xi^H\right\rangle \\
&=\left\langle \eta,\xi^H\right\rangle \left( \frac{3(2n-1)}{2n+1} H_M(\xi) -  l(\xi) \right).
\end{align*}
Keeping in mind Definition \ref{def: umbilical}, we obtain
$$\eta(\varphi_h)=\left\langle \eta,\xi^H\right\rangle \left( \frac{3(2n-2)}{2n+1} k(\xi) + \frac{2-2n}{2n+1}l(\xi) \right)=\frac{2n-2}{2n+1}\left\langle \eta,\xi^H\right\rangle \left( 3k(\xi)-l(\xi)\right)$$
as desired.
\end{proof}

We now show that in fact $\varphi_h\equiv 0$ and $l\equiv 3k$. There are two main tools in the proof: the use of the Codazzi equations found in \cite{CCHY18}, and the analysis of the global behaviour of the auxiliary function $|\xi^H|^{2n-2}\varphi_h(\xi)$ (a weighted version of $\varphi_h$). 

\begin{lemma}\label{lem2n}
Fix $n\geq 2$. Let $M$ be a smooth hypersurface in $\H^n$ which is connected and orientable. Assume that $M$ is umbilic, and suppose that there exists $c\neq 0$ such that $H_M(\xi)=c|\xi^H|$ for every point $\xi\in M\smallsetminus S_M$. Then, for all $\xi\in M\smallsetminus S_M$, we have
\begin{equation}\label{3props}
\begin{cases}
\left\langle \nu^H,\xi^H\right\rangle= |\xi^H|^2 k(\xi), \\
\left\langle \eta,\xi^H\right\rangle= 2|\xi^H|^2 \frac{\left\langle \nu_\xi, T\right\rangle}{|\mathcal{P}_H(\nu_\xi)|},\\
\left\langle V_j,\xi^H\right\rangle = \left\langle W_j,\xi^H\right\rangle=0\quad\mbox{ for }j\in\{1,\ldots,n-1\}.
\end{cases}
\end{equation}
If in addition $M$ is compact and without boundary, we have that $c>0$, $S_M = M\cap L_v\not =\emptyset$, and
\begin{equation}\label{2props}
\varphi_h\equiv 0 \equiv l- 3k.
\end{equation}
\end{lemma}
\begin{proof}
Let us divide the proof in multiple steps.\\
\noindent{\it Step I. } We first show the validity of \eqref{3props}. Let
$$
\alpha=\frac{2\left\langle \nu, T\right\rangle}{|\mathcal{P}_H\nu|}.
$$
By \cite[Proposition 4.2]{CCHY18} we know that
\begin{equation}\label{er}
V_j(k)=W_j(k)=0=V_j(l)=W_j(l)\quad\mbox{ for }j\in\{1,\ldots,n-1\}
\end{equation}
and
\begin{equation}\label{etakalpha}
\eta(k)=(l-2k)\alpha,\qquad \eta(\alpha)=k^2-\alpha^2-kl.
\end{equation}
For $\xi\in M\smallsetminus \left(S_M\cup L_v\right)$, from \eqref{er} and the identity $(2n-1)H_M=(2n-2)k + l$ we obtain
$$
\frac{\left\langle V_j,\xi^H\right\rangle}{|\xi^H|}=V_j(|\xi^H|)=\frac{1}{c}V_j(H_M(\xi))=\frac{2n-2}{c(2n-1)}V_j(k)+\frac{1}{c(2n-1)}V_j(l)=0
$$
for $j\in\{1,\ldots,n-1\}$. The same holds for $\left\langle W_j, \xi^H\right\rangle$. This shows that
\begin{equation}\label{VWxi}
\left\langle V_j,\xi^H\right\rangle = \left\langle W_j,\xi^H\right\rangle=0\quad\mbox{ for }j\in\{1,\ldots,n-1\}\mbox{ and }\xi\in M\smallsetminus S_M.
\end{equation}
We then deduce that the function $|\xi^H|$ is constant even along the commutators of the vector fields in ${\rm{span}}\left\{V_1,W_1,\ldots,V_{n-1},W_{n-1}\right\}$. Exploiting Lemma \ref{commutator} this implies 
$$\left\langle \tau, \xi^H\right\rangle=\frac{k(\xi)|\mathcal{P}_H(\nu)|}{2}\left\langle \eta,\xi^H\right\rangle.$$
Recalling that $\tau=\left\langle\nu, T \right\rangle \nu^H - |\mathcal{P}_H(\nu)| T$ and using $\left\langle T,\xi^H\right\rangle=0$ we infer
\begin{equation}\label{alphaperkappa}
\alpha(\xi) \left\langle \nu^H, \xi^H\right\rangle = k(\xi) \left\langle \eta, \xi^H\right\rangle \quad\mbox{ for }\xi\in M\smallsetminus S_M.
\end{equation}
From \eqref{eq cov der zeta xih}, \eqref{eq: relations J}, and the umbilicality condition in Definition \ref{def: umbilical}, we can compute
\begin{align}\label{1menol}
&\eta( \left\langle \eta, \xi^H\right\rangle )=1+ \left\langle \nabla_\eta \eta, \xi^H\right\rangle=1+ \left\langle \nabla_\eta J\eta, J\xi^H\right\rangle =1+ \left\langle \nabla_\eta \nu^H, J\xi^H\right\rangle\notag\\
&=1 + l(\xi)\left\langle \eta , J\xi^H\right\rangle=1- l(\xi)\left\langle \nu^H , \xi^H\right\rangle.
\end{align}
If we now differentiate the identity \eqref{alphaperkappa} along $\eta$, the relations \eqref{umbder}, \eqref{etakalpha}, and \eqref{1menol} yield
\begin{align*}
0&= \eta(\alpha)\left\langle \nu^H, \xi^H\right\rangle + \alpha \eta\left(\left\langle \nu^H, \xi^H\right\rangle\right) - \eta(k) \left\langle \eta, \xi^H\right\rangle - k \eta\left(\left\langle \eta, \xi^H\right\rangle\right)\\
&=(k^2-\alpha^2-kl)\left\langle \nu^H, \xi^H\right\rangle + \alpha l \left\langle \eta, \xi^H\right\rangle + (2k-l)\alpha \left\langle \eta, \xi^H\right\rangle + k l\left\langle \nu^H , \xi^H\right\rangle - k\\
&=(k^2-\alpha^2)\left\langle \nu^H, \xi^H\right\rangle + 2k\alpha \left\langle \eta, \xi^H\right\rangle - k\\
&=(k^2+\alpha^2)\left\langle \nu^H, \xi^H\right\rangle -k + 2\alpha (k \left\langle \eta, \xi^H\right\rangle - \alpha \left\langle \nu^H, \xi^H\right\rangle).
\end{align*}
Keeping in mind \eqref{alphaperkappa}, this says that
\begin{equation}\label{nuetakalpha}
\left\langle \nu^H, \xi^H\right\rangle = \frac{k(\xi)}{k^2(\xi)+\alpha^2(\xi)} \quad\mbox{ and }\quad \left\langle \eta, \xi^H\right\rangle = \frac{\alpha(\xi)}{k^2(\xi)+\alpha^2(\xi)}
\end{equation}
at least for any $\xi \in M\smallsetminus S_M$ where $k(\xi)\neq 0$. Notice that in our assumptions we have $k^2+\alpha^2>0$ in $M\smallsetminus S_M$ (see \cite[part (a) in Theorem B]{CCHY16}, and keep in mind that $\alpha\not\equiv 0$ due to the boundedness of $M$). Let also notice that, if $k$ vanishes at a point $\bar{\xi}\in M\smallsetminus (S_M\cup L_v)$, then $\eta(k)(\bar{\xi})=l(\bar{\xi})\alpha(\bar{\xi})=(2n-1)c|\bar{\xi}^H|\alpha(\bar{\xi})\neq 0$. Hence the relations \eqref{nuetakalpha} hold true by continuity throughout $M\smallsetminus S_M$. This implies that
\begin{align}\label{relakn}
&\frac{1}{k^2(\xi)+\alpha^2(\xi)}=\left\langle \nu^H, \xi^H\right\rangle^2 + \left\langle \eta, \xi^H\right\rangle^2\\
&=\left\langle \nu^H, \xi^H\right\rangle^2 + \left\langle \eta, \xi^H\right\rangle^2 + \sum_{j=1}^{n-1}\left\langle V_j, \xi^H\right\rangle^2 + \left\langle W_j, \xi^H\right\rangle^2=|\xi^H|^2,\notag
\end{align}
where in the second equality we used \eqref{VWxi}. Inserting the last identity in \eqref{nuetakalpha}, we get
\begin{equation}\label{nuetakalpha2}
\left\langle \nu^H, \xi^H\right\rangle = |\xi^H|^2 k(\xi) \quad\mbox{ and }\quad \left\langle \eta, \xi^H\right\rangle = |\xi^H|^2 \alpha(\xi)\qquad\mbox{ for }\xi\in M\smallsetminus S_M.
\end{equation}
The combination of \eqref{VWxi} and \eqref{nuetakalpha2} completes the proof of \eqref{3props}. In particular, since the function $\alpha^2\to\infty$ only at characteristic points, from \eqref{relakn} we can also deduce that
\begin{equation}\label{uguali}
S_M = M\cap L_v.
\end{equation}
As a matter of fact, the inclusion $M\cap L_v\subseteq S_M$ follows from the fact that at non-characteristic points $|\xi^H|^{-2}=k^2(\xi)+\alpha^2(\xi)$ is finite whereas the inclusion $S_M\subseteq M\cap L_v$ is a consequence of the boundedness of $\alpha^2(\xi)\leq k^2(\xi)+\alpha^2(\xi)=|\xi^H|^{-2}$ outside of $L_v$.\\
\noindent{\it Step II. } We now show that
$$\xi\mapsto \phi(\xi):=|\xi^H|^{2n-2}\varphi_h(\xi) \quad\mbox{ is constant throughout } M\smallsetminus S_M.$$
To this aim, using \eqref{nuetakalpha2} we can rewrite the function $\varphi_h$ in the following way
\begin{align}\label{phihl3k}
\varphi_h(\xi)&=\frac{2n-1}{2n+1}H_M(\xi)|\xi^H|^2 - k(\xi) |\xi^H|^2=|\xi^H|^2\left( \frac{(2n-2)k(\xi) + l(\xi) }{2n+1} - k(\xi) \right)\notag\\
&=\frac{l(\xi)-3k(\xi)}{2n+1}|\xi^H|^2,
\end{align}
so that
$$
\phi(\xi)=\frac{l(\xi)-3k(\xi)}{2n+1}|\xi^H|^{2n}.
$$
It is clear from \eqref{er} and \eqref{VWxi} that
$$
V_j(\phi)=W_j(\phi)=0 \qquad \mbox{ for all }j\in\{1,\ldots,n-1\},
$$
which also implies by Lemma \ref{commutator} that
$$\tau(\phi)-\frac{k|\mathcal{P}_H(\nu)|}{2}\,\eta(\phi)=0.$$
On the other hand, by using \eqref{phihl3k} and Lemma \ref{lem1n} we obtain
\begin{align*}
&\eta(\phi)=\eta\left( |\xi^H|^{2n-2}\varphi_h\right)\\
&=(2n-2)|\xi^H|^{2n-4}\left\langle \eta,\xi^H\right\rangle \varphi_h + |\xi^H|^{2n-2}\eta(\varphi_h)\\
&=\frac{2n-2}{2n+1}|\xi^H|^{2n-2}\left\langle \eta,\xi^H\right\rangle (l-3k) + \frac{2n-2}{2n+1} |\xi^H|^{2n-2}\left\langle \eta,\xi^H\right\rangle \left( 3k-l\right) =0.
\end{align*}
This says that the function $\phi$ is constant along every tangent vector fields in $M\smallsetminus S_M$ and concludes the proof of the current step.\\
\noindent{\it Step III. } From now on we shall assume that $M$ is also compact and without boundary. Let us show that
$$
c>0.
$$
We argue similarly to the proof of Step I in Lemma \ref{lem2}. By the compactness of $M$ the function $\frac{1}{2}|\xi^H|^2$ attains its maximum at a point $\xi_1\in M\smallsetminus L_v$, and we know from \eqref{uguali} that $\xi_1 \notin S_M$. Then, at $\xi=\xi_1$ we have
$$
0=\eta\left(\frac{1}{2}|\xi^H|^2\right)=\left\langle \eta, \xi_1^H\right\rangle\quad\mbox{and}\quad 0=\tau\left(\frac{1}{2}|\xi^H|^2\right)=\left\langle \nu,T\right\rangle \left\langle \nu^H, \xi_1^H\right\rangle.
$$ 
Since by \eqref{VWxi} we have 
$$\la \nu^H,\xi_1^H\ra^2=\la \eta,\xi_1^H\ra^2+\la \nu^H,\xi_1^H\ra^2 + \sum_{j=1}^n\la V_j,\xi_1^H\ra^2 + \la W_j,\xi_1^H\ra^2 = |\xi_1^H|^2>0,$$
we deduce that $\left\langle \nu,T\right\rangle=0$ and therefore
$$
k(\xi_1)=\frac{1}{|\xi^H_1|^2}\left\langle \nu^H, \xi_1^H\right\rangle = \frac{1}{|\xi^H_1|^2\,|\mathcal{P}_H \nu|} \left\langle \nu, \xi_1 \right\rangle >0
$$
where the positive sign is a consequence of the maximality condition and the fact that $\nu$ is the outward normal. Moreover, at the maximum point $\xi=\xi_1$ we obtain from \eqref{1menol}
$$
1-l(\xi_1) \left\langle \nu^H, \xi_1^H\right\rangle= \eta^2\left(\frac{1}{2}|\xi^H|^2\right)\leq 0
$$
which says that
$$
l(\xi_1)\geq \frac{1}{\left\langle \nu^H, \xi_1^H\right\rangle}>0.
$$
Therefore, keeping in mind Definition \ref{def: umbilical}, we have
$$
c=\frac{l(\xi_1)}{(2n-1)|\xi^H_1|}+\frac{(2n-2)k(\xi_1)}{(2n-1)|\xi^H_1|}>0.
$$
\noindent{\it Step IV. } We now show that
$$\phi(\xi)=|\xi^H|^{2n-2}\varphi_h(\xi)\equiv 0 \quad\mbox{ for }\xi\in M\smallsetminus S_M.$$
We already know from {\it Step II} that $\phi$ is identically equal to a constant value $\phi_0$. By contradiction we shall assume that $\phi_0\neq 0$. Since from the definition of $\varphi_h$ in \eqref{defphisn} it is clear that $\varphi_h$ and $\phi$ tend to $0$ as $\xi$ approaches $M\cap L_v$ and we know from \eqref{uguali} that $M\cap L_v=S_M$, we have that $S_M=M\cap L_v=\emptyset$ (so that there exist $0<r_m\leq r_M<\infty$ satisfying $r_m\leq |\xi^H|\leq r_M$). If we consider the integral curve $\gamma$ of $\eta$ starting from any point $\xi_0 \in M$, we can then extend $\gamma$ indefinitely. Arguing similarly to the proof of Step II in Lemma \ref{lem2} (from which we also borrow the analogous notations for the smooth functions $t(s)$, $r(s)$, and $\theta(s)$) we want to infer that the assumption $\phi_0\neq 0$ leads to the unboundedness of $\gamma(s)$ (which contradicts the compactness of $M$). We can rewrite 
$$
\frac{c(2n-1)}{(2n+1)}r^{2n+1}(s)-r^{2n-1}(s)\cos{\left(\theta(s)\right)}=\phi_0,
$$
which implies in particular that along the curve $\gamma$ one has $r(s)=R(\cos(\theta(s)))$ (i.e. the positive function $r(s)$ is uniquely determined by the value $\cos{\left(\theta(s)\right)}$). From \eqref{tderivative} we infer that
\begin{equation}\label{tprimen}
t'(s)=-2r(s)\cos{\left(\theta(s)\right)}=2\phi_0 r^{2-2n}(s)-\frac{2c(2n-1)}{(2n+1)}r^{3}(s).
\end{equation}
Since $c>0$ by {\it Step III} and $r(s)$ is bounded, if $\phi_0<0$ then $t'(s)\leq 2\phi_0 r_{M}^{2-2n}<0$ and $t(s)$ would be forced to be unbounded. We can then assume $\phi_0>0$. Exploiting \eqref{umbder}, \eqref{1menol}, and \eqref{VWxi} we recognize that
\begin{align*}
&\eta\left(\arctan\left(\frac{\left\langle \eta,\xi^H\right\rangle}{\left\langle \nu^H,\xi^H\right\rangle}\right)\right)=\frac{(1-l(\xi)\left\langle \nu^H,\xi^H\right\rangle)\left\langle \nu^H,\xi^H\right\rangle - l(\xi) \left\langle \eta,\xi^H\right\rangle^2 }{\left\langle \eta,\xi^H\right\rangle^2+\left\langle \nu^H,\xi^H\right\rangle^2}\\
&=\frac{\left\langle \nu^H,\xi^H\right\rangle-l(\xi) |\xi^H|^2}{|\xi^H|^2}= \frac{\left\langle \nu^H,\xi^H\right\rangle-(2n+1)\varphi_h(\xi)-3k(\xi)|\xi^H|^2}{|\xi^H|^2}\\
&=\frac{-2\left\langle \nu^H,\xi^H\right\rangle-(2n+1)\varphi_h(\xi)}{|\xi^H|^2},
\end{align*}
where in the last two equalities we have made use of \eqref{phihl3k} and \eqref{nuetakalpha2}. The previous identity yields
\begin{align}\label{thetaprimen}
\theta'(s)&=\frac{2n-1}{r^2(s)}\left(r(s)\cos{\left(\theta(s)\right)}-c r^3(s)\right)\\
&=\frac{2n-1}{r^2(s)}\left(-\phi_0 r^{2-2n}(s) -\frac{2c}{2n+1} r^3(s)\right).\notag
\end{align}
Since we know that $c>0$, the assumption $\phi_0>0$ implies that $\theta(s)$ is strictly decreasing and $\theta(s)\to -\infty$ as $s\to \infty$. We can then pick a strictly increasing sequence of values $\{s_k\}_{k\in\N}$ such that $\theta(s_k)-\theta(s_{k+1})=2\pi$ for all $k\in\N$. From \eqref{tprimen} and \eqref{thetaprimen} we get
\begin{align*}
&t(s_{k+1})-t(s_k)=\int_{s_k}^{s_{k+1}} t'(s) ds \\
&= -2 \int_{s_k}^{s_{k+1}} r(s)\cos(\theta(s)) ds=\frac{2}{2n-1} \int_{s_k}^{s_{k+1}} \frac{r^3(s)\cos(\theta(s))}{\phi_0 r^{2-2n}(s) +\frac{2c}{2n+1} r^3(s)} \theta'(s) ds\\
&= \frac{2}{(2n-1)c} \int_{s_k}^{s_{k+1}} \frac{(cr^3(s) - r(s)\cos(\theta(s)) + r(s)\cos(\theta(s)))\cos(\theta(s))}{\phi_0 r^{2-2n}(s) +\frac{2c}{2n+1} r^3(s)} \theta'(s) ds\\
&=  \frac{2}{(2n-1)c} \int_{s_k}^{s_{k+1}}\cos(\theta(s))\theta'(s) ds + \frac{2}{(2n-1)c} \int_{s_{k}}^{s_{k+1}}  \frac{r(s)\cos^2(\theta(s))}{\phi_0 r^{2-2n}(s) +\frac{2c}{2n+1} r^3(s)} \theta'(s) ds\\
&=-\frac{2}{(2n-1)c} \int_{\theta(s_{k+1})}^{\theta(s_{k})}  \frac{R(\cos(\sigma))\cos^2(\sigma)}{\phi_0 R^{2-2n}(\cos(\sigma)) +\frac{2c}{2n+1} R^3(\cos(\sigma))} d\sigma\\
&=-\frac{2}{(2n-1)c} \int_{\theta(s_{1})-2\pi}^{\theta(s_{1})}   \frac{R(\cos(\sigma))\cos^2(\sigma)}{\phi_0 R^{2-2n}(\cos(\sigma)) +\frac{2c}{2n+1} R^3(\cos(\sigma))} d\sigma
\end{align*}
for every $k\in\N$. Since the term $t(s_{k+1})-t(s_k)$ is strictly negative and independent of $k$, we conclude as in Lemma \ref{lem2} that $t(s_k)\to - \infty$ as $k\to \infty$. Hence we have reached a contradiction in both scenarios $\phi_0<0$ and $\phi_0>0$. This ensures the validity of $\phi\equiv 0$.\\
\noindent{\it Step V. } In this final step we finish the proof of the desired statements. Keeping the same notations as before, if we insert the information $\phi\equiv 0$ proved in {\it Step IV} in \eqref{tprimen} we infer
\begin{equation}\label{tprimenn}
t'(s)=-\frac{2c(2n-1)}{(2n+1)}r^{3}(s)\leq 0.
\end{equation}
If the sets $S_M$ and $M\cap L_v$ were empty, we could extend the integral curves $\gamma$ of $\eta$ indefinitely and we would have the contradicting property $t(s)\to -\infty$ as $s\to \infty$. Therefore, by \eqref{uguali}, it has to be
\begin{equation}\label{ugualinonvuoti}
S_M=M\cap L_v\not =\emptyset.
\end{equation}
Finally, exploiting again \eqref{uguali} and the identity $\phi\equiv 0$ in $M\smallsetminus S_M$, we have
$$
\varphi_h\equiv 0 \qquad\mbox{ in } M \smallsetminus S_M
$$
and, by \eqref{phihl3k}, also
$$
l-3k \equiv 0 \qquad\mbox{ in } M \smallsetminus S_M.
$$
This concludes the proof of \eqref{2props}, and of the lemma.
\end{proof}

We are finally ready to complete the proof of Theorem \ref{mainumb}.

\begin{proof}[Proof of Theorem \ref{mainumb}]
With Lemma \ref{lem1n} and Lemma \ref{lem2n} in hand, in order to complete the proof we can follow closely the arguments of Theorem \ref{main}. In fact, we deduce from \eqref{ugualinonvuoti} that there exists $p\in\N$ such that
$$
S_M=\{(0,0,t_1),\ldots,(0,0,t_p)\}
$$ 
for some $t_1<t_2<\ldots<t_p$ (we stress that points in $S_M=M\cap L_v$ cannot accumulate since $T$ is aligned with the normal direction at characteristic points). If we consider any point $\xi_0 \in M\cap \{t<t_2\}$ such that $\xi_0\not =(0,0,t_1)$, we can look at the integral curve $\gamma$ of $\eta$ starting from $\xi_0$. With the same notations as in Lemma \ref{lem2n}, we know from \eqref{tprimenn} that $\gamma \subset M\cap \{t<t_2\}$ and 
$$\gamma(s) \mbox{ reaches }(0,0,t_1).$$ 
As $r(s)>0$ and it is reaching $0$, we obtain from the identity $\varphi_h\equiv 0$ proved in Lemma \eqref{lem2} that
$$
\cos(\theta(s))=\frac{(2n-1)c}{2n+1} r^2(s)>0\qquad\mbox{ and }\qquad \cos(\theta(s)) \to 0.
$$
Recalling \eqref{thetaprimen}, this implies that $\sin(\theta(s))$ is decreasing and
$$
\sin(\theta(s))\to -1.
$$
This yields
\begin{equation}\label{limphiv}
\frac{(2n-1)c}{2n+1}t(s)-\sin(\theta(s)) \to \frac{(2n-1)c}{2n+1}t_1+1\quad\mbox{ as $\gamma(s)$ approaches }(0,0,t_1).
\end{equation}
On the other hand, if we differentiate along $\eta$ the function
$$
\varphi_v(\xi)=\frac{2n-1}{2n+1}H_M(\xi) \frac{t}{|\xi^H|} - \frac{\la \eta,\xi^H\ra}{|\xi^H|}= \frac{(2n-1)c}{2n+1}t - \frac{\la \eta,\xi^H\ra}{|\xi^H|},
$$
for $\xi\in M\smallsetminus S_M$, by using \eqref{tderivative} and \eqref{1menol}, we obtain
\begin{align*}
&\eta\left(\varphi_v\right)(\xi)=-\frac{2(2n-1)c}{2n+1}\left\langle \nu^H,\xi^H\right\rangle - \frac{1-l\left\langle \nu^H,\xi^H\right\rangle}{|\xi^H|}+\frac{\la \eta,\xi^H\ra^2}{|\xi^H|^3}\\
&=\frac{1}{|\xi^H|^3}\left( -\frac{2(2n-1)H_M}{2n+1} |\xi^H|^2\left\langle \nu^H,\xi^H\right\rangle -|\xi^H|^2 + l |\xi^H|^2\left\langle \nu^H,\xi^H\right\rangle + \left\langle \eta,\xi^H\right\rangle^2 \right)\\
&=\frac{\left\langle \nu^H,\xi^H\right\rangle}{|\xi^H|^3}\left( -\frac{2(2n-1)H_M}{2n+1} |\xi^H|^2 -\left\langle \nu^H,\xi^H\right\rangle + l |\xi^H|^2\right)+\\
&-\frac{ \sum_{j=1}^{n-1} \left\langle V_j,\xi^H\right\rangle^2+\left\langle W_j,\xi^H\right\rangle^2}{|\xi^H|^3}.
\end{align*}
We can now use the properties \eqref{3props}-\eqref{2props} established in Lemma \ref{lem2n} together with \eqref{l3k}, and we deduce
\begin{align*}
\eta\left(\varphi_v\right)(\xi)=\frac{\left\langle \nu^H,\xi^H\right\rangle}{|\xi^H|^3}\left( -2k(\xi) |\xi^H|^2 -\left\langle \nu^H,\xi^H\right\rangle + 3k(\xi) |\xi^H|^2\right)=0.
\end{align*}
Hence $\varphi_v$ is constant along $\gamma$. From \eqref{limphiv} we know that such constant has to be equal to $\frac{(2n-1)c}{2n+1}t_1+1$. The arbitrariness of $\xi_0\in \left(M\cap \{t<t_2\}\right)\smallsetminus\{(0,0,t_1)\}$ (which is the starting point of $\gamma$) yields
$$
\varphi_v \equiv \frac{(2n-1)c}{2n+1}t_1+1 \qquad\mbox{ in }\left(M\cap \{t<t_2\}\right)\smallsetminus\{(0,0,t_1)\}.
$$ 
The previous identity and the identity $\varphi_h\equiv 0$ can be rewritten, keeping in mind the definitions in \eqref{defphisn}, as
$$
\left\langle \nu^H,\frac{\xi^H}{|\xi^H|}\right\rangle=\frac{(2n-1)c}{2n+1}|\xi^H|^2\quad\mbox{ and }\quad \left\langle \eta,\frac{\xi^H}{|\xi^H|}\right\rangle=\frac{(2n-1)c}{2n+1}\left(t-t_1-\frac{2n+1}{(2n-1)c}\right).
$$
This implies
\begin{align*}
1&=\left\langle \nu^H,\frac{\xi^H}{|\xi^H|}\right\rangle^2+\left\langle \eta,\frac{\xi^H}{|\xi^H|}\right\rangle^2\\
&= \left(\frac{(2n-1)c}{2n+1}|\xi^H|^2\right)^2+\left(\frac{(2n-1)c}{2n+1}\left(t-t_1-\frac{2n+1}{(2n-1)c}\right)\right)^2
\end{align*}
for any $\xi \in (M\cap \{t<t_2\})\smallsetminus\{(0,0,t_1)\}$, which shows that
$$
(M\cap \{t<t_2\})\smallsetminus\{(0,0,t_1)\} \subset \de B_R(0,0,t_0)
$$
with
$$
R^2=\frac{2n+1}{(2n-1)c}\qquad\mbox{ and }\qquad t_0=t_1+\frac{2n+1}{(2n-1)c}.
$$
This allows, as in Theorem \ref{main}, to conclude the proof of the desired statement. 
\end{proof}

\subsection{The axially symmetric case}

As a concrete application of our main theorems we want to single out a relevant class of hypersurfaces in which we have a uniqueness result for the gauge spheres. We already mentioned in the Introduction (see also \cite{HL, vitt, Tpd, FM, MTj, HZ} for related settings) that it is quite typical to require some apriori symmetry in terms of rotational invariances. More precisely, we can recall the following well-known class of symmetric domains (which is consistent with the type of prescription of the horizontal curvature $H_M$ we are dealing with).

\begin{definition}\label{defcyli}
For $n\geq 1$ we say that a smooth hypersurface $M\subset\Hn$ is cylindrically symmetric if, locally around any point of $M$, there exists a defining function $f$ for $M$ which can be written as
\begin{equation}\label{defcyl}
f(x,y,t)=v(|x|^2+|y|^2,t)
\end{equation}
for some smooth function $v$.
\end{definition}

We have the following

\begin{corollary}\label{corcyl}
Fix $n\geq 1$. Let $M$ be a smooth hypersurface of $\Hn$ which is connected, orientable, compact, and without boundary. Suppose that $M$ is cylindrically symmetric with respect to Definition \ref{defcyli} and that, at every non-characteristic point $(x,y,t)\in M$, the horizontal mean curvature of $M$ is proportional to $\sqrt{|x|^2+|y|^2}$ up to a constant factor $c\neq 0$. Then $c>0$ and there exists $t_0\in\R$ such that $M=\de B_R(\xi_0)$ with $R=\sqrt{\frac{1}{c}\frac{2n+1}{2n-1}}$ and $\xi_0=(0,0,t_0)$.
\end{corollary}
\begin{proof}
Fix an open neighborhood $U\subset \Hn$ where $U\cap M$ is described, as in \eqref{defcyl}, by the zero-level set of a smooth function $f$ with non-null gradient. Pick the sign of $f$ such that the outward normal $\nu$ at $\xi\in U\cap M$ is equal to
$$
\nu=\frac{Tf T + \sum_{j=1}^{n} X_jf X_j + Y_j f Y_j}{\left( (Tf)^2+\sum_{j=1}^{n} (X_jf)^2 + (Y_j f)^2 \right)^{\frac{1}{2}}}.
$$
By exploiting \eqref{defcyl} one has $Tf=v_2$, $X_j f= 2x_j v_1 -2y_j v_2$ and $Y_j f= 2y_j v_1 + 2x_i v_2$, which implies that
$$
|\mathcal{P}_H\nu|^2=\frac{4(|x|^2+|y|^2)(v^2_1+v^2_2)}{4(|x|^2+|y|^2)(v^2_1+v^2_2) + v_2^2}.
$$
This is saying that for cylindrically symmetric hypersurfaces we can have characteristic points only when $|x|^2+|y|^2=0$, i.e.
$$
S_M\subseteq M\cap L_v.
$$
Therefore, in case $n=1$ we can apply Theorem \ref{main} to infer the desired statement.\\
Fix then $n\geq 2$. We want to check that the cylindrically symmetric assumption implies that $M$ is in fact horizontally umbilical. Since we have
$$
\nu^H=\sum_{j=1}^n \frac{x_jv_1-y_jv_2}{\sqrt{|x|^2+|y|^2}\sqrt{v_1^2+v_2^2}} X_j + \frac{y_jv_1+x_jv_2}{\sqrt{|x|^2+|y|^2}\sqrt{v_1^2+v_2^2}} Y_j
$$
and
$$
\eta=\sum_{j=1}^n \frac{y_jv_1+x_jv_2}{\sqrt{|x|^2+|y|^2}\sqrt{v_1^2+v_2^2}} X_j + \frac{y_jv_2-x_jv_1}{\sqrt{|x|^2+|y|^2}\sqrt{v_1^2+v_2^2}} Y_j,
$$
for any $j,k\in\{1,\ldots,n\}$ we can directly compute
\begin{align*}
\la\na_{X_k}\nu^H,X_j\ra=X_k(\left\langle \nu^H,X_j\right\rangle)
=\frac{\delta_{jk}v_1 + 2x_j(x_k v_{11}-y_kv_{12})-2y_j(x_k v_{12} -y_k v_{22})}{\sqrt{|x|^2+|y|^2}\sqrt{v_1^2+v_2^2}}+&\\
-\left\langle \nu^H,X_j\right\rangle \left( \frac{x_k}{|x|^2+|y|^2}+\frac{2v_1(x_k v_{11}-y_k v_{12})+2v_2(x_k v_{12}-y_k v_{22})}{v_1^2+v_2^2} \right),&\\
\la\na_{X_k}\nu^H,Y_j\ra=X_k(\left\langle \nu^H,Y_j\right\rangle)=\frac{\delta_{jk}v_2 + 2x_j(x_k v_{12}-y_kv_{22})+2y_j(x_k v_{11} -y_k v_{12})}{\sqrt{|x|^2+|y|^2}\sqrt{v_1^2+v_2^2}}+&\\
-\left\langle \nu^H,Y_j\right\rangle \left( \frac{x_k}{|x|^2+|y|^2}+\frac{2v_1(x_k v_{11}-y_k v_{12})+2v_2(x_k v_{12}-y_k v_{22})}{v_1^2+v_2^2} \right),&\\
\la\na_{Y_k}\nu^H,X_j\ra=Y_k(\left\langle \nu^H,X_j\right\rangle)=\frac{-\delta_{jk}v_2 + 2x_j(y_k v_{11}+x_kv_{12})-2y_j(y_k v_{12} +x_k v_{22})}{\sqrt{|x|^2+|y|^2}\sqrt{v_1^2+v_2^2}}+&\\
-\left\langle \nu^H,X_j\right\rangle \left( \frac{y_k}{|x|^2+|y|^2}+\frac{2v_1(y_k v_{11}+x_k v_{12})+2v_2(y_k v_{12}+x_k v_{22})}{v_1^2+v_2^2} \right),&\\
\la\na_{Y_k}\nu^H,Y_j\ra=Y_k(\left\langle \nu^H,Y_j\right\rangle)=\frac{\delta_{jk}v_1 + 2x_j(y_k v_{12}+x_kv_{22})+2y_j(y_k v_{11} +x_k v_{12})}{\sqrt{|x|^2+|y|^2}\sqrt{v_1^2+v_2^2}}+&\\
-\left\langle \nu^H,Y_j\right\rangle \left( \frac{y_k}{|x|^2+|y|^2}+\frac{2v_1(y_k v_{11}+x_k v_{12})+2v_2(y_k v_{12}+x_k v_{22})}{v_1^2+v_2^2} \right).&
\end{align*}
From the previous relations we can recognize by a straightforward computation that
$$
\mathcal{P}_H(\na_Z\nu^H) - \frac{2\left\langle \nu, T\right\rangle}{|\mathcal{P}_H(\nu)|} J(Z)=k Z \quad\mbox{ for any }Z\in \mathcal{H}\mbox{ such that }\left\langle Z,\nu^H\right\rangle=\left\langle Z,\eta\right\rangle=0
$$
and
$$
\mathcal{P}_H(\na_\eta\nu^H)=l\eta
$$
where
$$
k=\frac{v_1}{\sqrt{|x|^2+|y|^2}\sqrt{v_1^2+v_2^2}},\qquad \mbox{and} 
$$
$$
l=\frac{v_1}{\sqrt{|x|^2+|y|^2}\sqrt{v_1^2+v_2^2}} + \frac{2\sqrt{|x|^2+|y|^2}}{(v_1^2+v_2^2)^{\frac{3}{2}}}\left(v_{11} v_2^2+v_{22}v_1^2 - 2v_{12} v_1v_2\right).
$$
A direct comparison with Definition \ref{defshape} and Definition \ref{def: umbilical} tells us that $M$ is umbilic. We can then apply Theorem \ref{mainumb} and complete the proof of the corollary.
\end{proof}


\begin{thebibliography}{10}


\bibitem{A} A. D. Aleksandrov, \textit{Uniqueness theorems for surfaces in the large {I}}, Vestnik Leningrad. Univ. 11 (1956) 5--17.

\bibitem{BTV}
Z. M. Balogh, J. T. Tyson, E. Vecchi, \textit{Intrinsic curvature of curves and surfaces and a Gauss-Bonnet theorem in the Heisenberg group}, Math. Z. 287 (2017) 1--38.

\bibitem{BBCH}
D. Barilari, U. Boscain, D. Cannarsa, K. Habermann, \textit{Stochastic processes on surfaces in three-dimensional contact sub-Riemannian manifolds}, Ann. Inst. Henri Poincar\'{e} Probab. Stat. 57 (2021) 1388--1410.

\bibitem{BLU}
A. Bonfiglioli, E. Lanconelli, F. Uguzzoni, \textit{Stratified Lie groups and potential theory for their sub-Laplacians}, Springer Monographs in Mathematics (2007), Springer, Berlin.

\bibitem{CDG}
L. Capogna, D. Danielli, N. Garofalo, \textit{The geometric Sobolev embedding for vector fields and the isoperimetric inequality}, Comm. Anal. Geom. 2 (1994) 203--215.

\bibitem{Cap} L. Capogna, \textit{Isoperimetric inequalities in the Heisenberg group and in the plane}, in ``Subelliptic {PDE}'s and applications to geometry and finance'', Lect. Notes Semin. Interdiscip. Mat. 6 (2007) 93--106.

\bibitem{CHMY}
J.-H. Cheng, J.-F. Hwang, A. Malchiodi, P. Yang,
\textit{Minimal surfaces in pseudohermitian geometry},
Ann. Sc. Norm. Super. Pisa Cl. Sci. (5) 4 (2005) 129--177.
	
\bibitem{CCHY16}
J.-H. Cheng, H.-L. Chiu, J.-F. Hwang, P. Yang,
\textit{Umbilic hypersurfaces of constant sigma-k curvature in the Heisenberg group},
Calc. Var. Partial Differential Equations 55 (2016) 66.

\bibitem{CCHY18}
J.-H. Cheng, H.-L. Chiu, J.-F. Hwang, P. Yang,
\textit{Umbilicity and characterization of Pansu spheres in the Heisenberg group},
J. Reine Angew. Math. 738 (2018) 203--235.

\bibitem{CL}
H.-L. Chiu, S.-H. Lai,
\textit{The fundamental theorem for hypersurfaces in Heisenberg groups},
Calc. Var. Partial Differential Equations 54 (2015) 1091--1118.

\bibitem{DGN03}
D. Danielli, N. Garofalo, D.-M. Nhieu, \textit{Notions of convexity in Carnot groups}, Comm. Anal. Geom. 11 (2003) 263--341.

\bibitem{DGN07}
D. Danielli, N. Garofalo, D.-M. Nhieu, \textit{Sub-Riemannian calculus on hypersurfaces in Carnot groups}, Adv. Math. 215 (2007) 292--378.

\bibitem{DGN08}
D. Danielli, N. Garofalo, D.-M. Nhieu, \textit{A partial solution of the isoperimetric problem for the Heisenberg group}, Forum Math. 20 (2008) 99--143. 

\bibitem{D}
G. Darboux, \textit{Le\c{c}ons sur la th\'{e}orie g\'{e}n\'{e}rale des surfaces. {IV}}, Les Grands Classiques Gauthier-Villars (1993), reprint of the 1896 original, \'{E}ditions Jacques Gabay, Sceaux.

\bibitem{Fo73}
G. B. Folland, \textit{A fundamental solution for a subelliptic operator}, Bull. Amer. Math. Soc. 79 (1973) 373--376.

\bibitem{FM}
V. Franceschi, R. Monti, \textit{Isoperimetric problem in {$H$}-type groups and Grushin spaces}, Rev. Mat. Iberoam. 32 (2016) 1227--1258.

\bibitem{Ga77}
B. Gaveau, \textit{Principe de moindre action, propagation de la chaleur et estim{\'e}es sous elliptiques sur certains groupes nilpotents}, Acta Math. 139 (1977) 95--153.

\bibitem{GM} C. Guidi, V. Martino, \textit{Horizontal curvatures and classification results}, Ann. Acad. Sci. Fenn. Math. 45 (2020) 829--840.

\bibitem{HZ} G. He, P. Zhao, \textit{The isoperimetric problem in the Heisenberg group {$\Hn$} with density}, Anal. Math. Phys. 10 (2020) 24.

\bibitem{HP} R. K. Hladky, S. D. Pauls, \textit{Constant mean curvature surfaces in sub-Riemannian geometry}, J. Differential Geom. 79 (2008) 111--139.

\bibitem{HL} J. Hounie, E. Lanconelli, \textit{An Alexandrov type theorem for Reinhardt domains of {${\mathbb C}^2$}}, in ``Recent progress on some problems in several complex variables and partial differential equations'', Contemp. Math. 400 (2006) 129--146. 

\bibitem{KV}
A. Kor\'anyi, S. V\'agi, \textit{Singular integrals on homogeneous spaces and some problems of classical analysis}, Ann. Scuola Norm. Sup. Pisa Cl. Sci. (3) 25 (1971) 575--648.

\bibitem{L13}
E. Lanconelli, \textit{``Potato kugel'' for sub-Laplacians}, Israel J. Math. 194 (2013) 277--283.

\bibitem{LM}
G. P. Leonardi, S. Masnou, \textit{On the isoperimetric problem in the Heisenberg group {$\Hn$}}, Ann. Mat. Pura Appl. (4) 184 (2005) 533--553.

\bibitem{vitt} V. Martino, \textit{A symmetry result on Reinhardt domains}, Differential Integral Equations 24 (2011) 495--504.

\bibitem{MTsigma} V. Martino, G. Tralli, \textit{High-order Levi curvatures and classification results}, Ann. Global Anal. Geom. 46 (2014) 351--359.

\bibitem{MTj} V. Martino, G. Tralli, \textit{A Jellett type theorem for the Levi curvature}, J. Math. Pures Appl. (9) 108 (2017) 869--884.

\bibitem{MTsw} V. Martino, G. Tralli, \textit{Overdetermined problems for gauge balls in the Heisenberg group}, preprint.

\bibitem{Mo} R. Monti, \textit{Heisenberg isoperimetric problem. The axial case}, Adv. Calc. Var. 1 (2008) 93--121.

\bibitem{MoRi} R. Monti, M. Rickly, \textit{Convex isoperimetric sets in the Heisenberg group}, Ann. Sc. Norm. Super. Pisa Cl. Sci. (5) 8 (2009) 391--415.

\bibitem{MoRos}
S. Montiel, A. Ros, \textit{Curves and surfaces}, Graduate Studies in Mathematics 69 (2005), American Mathematical Society, Providence, RI.

\bibitem{Ni} Y. Ni, \textit{Sub-Riemannian constant mean curvature surfaces in the Heisenberg group as limits}, Ann. Mat. Pura Appl. (4) 183 (2004) 555--570.

\bibitem{PP}
P. Pansu, \textit{Une in\'egalit\'e isop\'erim\'etrique sur le groupe de Heisenberg}, C. R. Acad. Sci. Paris S\'er. I Math. 295 (1982) 127--130.

\bibitem{P}
S. D. Pauls, \textit{Minimal surfaces in the Heisenberg group}, Geom. Dedicata 104 (2004) 201--231.

\bibitem{RiRo}
M. Ritor{\'e}, C. Rosales, \textit{Rotationally invariant hypersurfaces with constant mean curvature in the Heisenberg group {$\Hn$}}, J. Geom. Anal. 16 (2006) 703--720.

\bibitem{RiRo1}
M. Ritor{\'e}, C. Rosales, \textit{Area-stationary surfaces in the Heisenberg group {$\mathbb{H}^1$}}, Adv. Math. 219 (2008) 633--671.

\bibitem{Ri}
M. Ritor\'e,
\textit{A proof by calibration of an isoperimetric inequality in the Heisenberg group $\H^n$},
Calc. Var. Partial Differential Equations 44 (2012) 47--60.

\bibitem{Tpd}
G. Tralli, \textit{Levi Curvature with Radial Simmetry: a Sphere Theorem for Bounded Reinhardt Domains of {$\mathbb{C}^{2}$}}, Rend. Sem. Mat. Univ. Padova 124 (2010) 185--196.

\end{thebibliography}
\end{document}